\documentclass[12pt,reqno]{article}
\usepackage{amssymb,amscd,amsmath,amsthm,color}
\newtheorem{theorem}{Theorem}[section]
\newtheorem{lemma}[theorem]{Lemma}
\newtheorem{cor}[theorem]{Corollary}
\newtheorem{prop}[theorem]{Proposition}
\usepackage{enumerate,amssymb}
\theoremstyle{definition}
\newtheorem{definition}[theorem]{Definition}
\newtheorem{example}[theorem]{Example}

\theoremstyle{remark}
\newtheorem{remark}[theorem]{Remark}

\numberwithin{equation}{section}

\def\bE{\mathbb{E}}
\def\bM{\mathbb{M}}
\def\bN{\mathbb{N}}
\def\bR{\mathbb{R}}

\def\bZ{\mathbb{Z}}
\def\ffi{\mathcal{\varphi}}
\def\bP{\mathbb{P}}
\def\bw{\mathbf{w}}
\def\bA{\mathbf{A}}
\def\Tr{\mathrm{Tr}\,}

\begin{document}
\baselineskip=15pt

\title{Jensen and Minkowski  inequalities for \qquad  \qquad \qquad operator means
and anti-norms}

\author{Jean-Christophe Bourin and Fumio Hiai} 

\date{ }

\maketitle

\begin{abstract}
\noindent
Jensen inequalities for positive linear maps of Choi and Hansen-Pedersen type are
established for a large class of operator/matrix means such as some $p$-means and some
Kubo-Ando means. These results are also extensions of the Minkowski determinantal
inequality. To this end we develop the study of anti-norms, a notion parallel to the 
symmetric norms in matrix analysis, including functionals like Schatten $q$-norms for a
parameter $q\in[-\infty,1]$ and the Minkowski functional $\det^{1/n}A$.  An interpolation theorem for the Schur multiplication is given in this setting.



\end{abstract}

\section{Introduction}
Jensen inequalities for matrices and operators have various versions. The most general ones
involve a unital positive linear map $\bE :\bM_n\to\bM_m$. For instance, if $f(t)$ is
operator concave on an interval $\Omega$, then
\begin{equation}\label{Choi}
f(\bE(Z)) \ge \bE(f(Z))
\end{equation}
for all $Z\in\bM_n\{\Omega\}$, the Hermitians with spectra in $\Omega$. This is Choi's
inequality \cite{Choi}, which is specialized to Hansen-Pedersen's inequality \cite{HP}
\begin{equation}\label{HP}
f\left(\sum_{i=1}^kC_i^*Z_iC_i\right) \ge \sum_{i=1}^k C_i^*f(Z_i)C_i
\end{equation}
for $C^*$-convex combinations in $\bM_n\{\Omega\}$ with $n\times m$ matrices $C_i$ such
that $\sum_{i=1}^k C_i^*C_i=I$, the identity. These Jensen's inequalities are famous
characterizations of operator concavity of the function $f$:
\begin{equation}\label{OC}
f\left(\frac{A+B}{2}\right) \ge \frac{f(A)+f(B)}{2},\qquad A,B\in\bM_n\{\Omega\}.
\end{equation}
Are there similar inequalities by making use of the $p$th power map 
$\bE_p(Z):=\bE^{1/p}(Z^p)$ with $p>0$\,? We will deal with this question in Section 2.
This contains some Jensen type inequalities for the power $p$-means
\begin{equation}\label{p-means}
A\,\beta_p\,B :=\left(\frac{A^p+B^p}{2}\right)^{1/p}
\end{equation}
of two positive operators $A,B$.

Sections 3 and 4 are concerned with the operator means in the Kubo-Ando sense \cite{KA}.
The concavity results obtained in Section 2 for the means \eqref{p-means} have analogous
statements for a natural class of operator means; this is the central part of the paper.
In Section 3 we obtain the Minkowski type inequality
\begin{equation}\label{Mink}
{\det}^{1/n}(A\,\sigma\,B)\ge({\det}^{1/n}A)\,\sigma\,({\det}^{1/n}B),
\end{equation}
when $\sigma$ is an operator mean with some geometric convexity property, in particular, an
average of the weighted geometric means $A\,\#_\alpha\,B$, which we will call a
geodesic mean. Thus \eqref{Mink} extends the  Minkowski inequality for the arithmetic mean.
Section 4 further extends these inequalities to those involving concave functions in the
general setting of anti-norms, a class of functionals on $\bM_n^+:=\bM_n\{[0,\infty)\}$,
including the Schatten $q$-anti-norms for $q\in(-\infty,1]$ and the Minkowski
functional $A\mapsto\det^{1/n}A$. Jensen type inequalities similar to those in Section 2
will be obtained for anti-norms.

The means in Sections 3 and 4 do not cover a wide class of Kubo-Ando means, but they turn
out rather natural as they have extensions for several variables, generalizing the 
geometric means of several matrices introduced by Moakher \cite{Mo} and Bhatia-Holbrook
\cite{BH} (also by \cite{ALM} in a different approach). This is our concern in Section 5.
We will extend some recent inequalities due to Lawson-Lim \cite{LL} and Bhatia-Karandikar
\cite{BK}.

Section 6, a related but independent complement, gives several basic facts on symmetric
anti-norms. It is noticed that the Minkowski functional $A\mapsto\det^{1/n}A$ is quite a
special anti-norm. We show some interpolation properties for symmetric
anti-norms, with a stronger version for Schur multiplication maps.
Finally, we point out a reverse H\"older inequality. 

Let $A,B\in\bM_n^+$, and let $\lambda_1(A)\ge\dots\ge\lambda_n(A)$ denote the
eigenvalues of $A$ listed in decreasing order with multiplicities. The supermajorization
$A\prec^w B$ means that
$$
\sum_{j=1}^k \lambda_{n+1-j}(A) \ge \sum_{j=1}^k \lambda_{n+1-j}(B),
\qquad k=1,\dots, n.
$$
If equality holds when $k=n$, we have the usual majorization $A\prec B$. We write
$A^{\downarrow}$ for the diagonal matrix whose entries on the diagonal are the
$\lambda_i(A)$'s in decreasing order, and $A^{\uparrow}$ for that whose diagonals are the
$\lambda_i(A)$'s in increasing order. The famous Lidskii-Wielandt and the Ky Fan
majorizations (see \cite{MO,Bh,LM}) are written as
\begin{equation}\label{Fan-Lidskii}
A^{\downarrow}+B^{\uparrow}\prec A+B\prec A^{\downarrow}+B^{\downarrow}.
\end{equation}
By the log-supermajorization $A\prec^{w(\log)}B$ we mean that
$$
\prod_{j=1}^k \lambda_{n+1-j}(A) \ge \prod_{j=1}^k \lambda_{n+1-j}(B), \qquad k=1,\dots,n.
$$
The log-supermajorization version of \eqref{Fan-Lidskii} for an operator mean $\sigma$
might be
\begin{equation}\label{Conj}
A^{\downarrow}\,\sigma\,B^{\uparrow}\prec^{w(\log)}
A\,\sigma\,B\prec^{w(\log)} A^{\downarrow}\,\sigma\,B^{\downarrow}.
\end{equation}
Although the problem of characterizing $\sigma$ for which two relations in \eqref{Conj}
hold is left open, we prove a partial result when $\sigma$ is a geodesic mean.

Two significant features of the present paper are continued from the previous
\cite{BH1}. The first is the relation between supermajorization and anti-norms. We noted
in \cite{BH1} that supermajorization leads to inequalities for anti-norms. In Section 4
we adapt this to log-supermajorization and a sub-class of anti-norms, called derived
anti-norms, and extend the Minkowski type inequalities in Section 3 to anti-norm
inequalities.

The second feature is the use of the Minkowski or Jensen type inequalities via unitary
orbits for concave functions. Likewise in \cite{BH1}, we apply the following substitute
for \eqref{Choi}--\eqref{OC} when $f$ is a general concave function.

\begin{theorem}\label{T-1.1}
Let $\bE : \bM_n\to \bM_m$ be a unital positive linear map, let $f(t)$ be a concave
function on an interval $\Omega$, and let $Z\in \bM_n\{\Omega\}$. Then, for some
unitaries $U,\,V\in\bM_m$,
\begin{equation*}
f(\bE(Z))\ge \frac{U\bE(f(Z))U^*+V\bE(f(Z))V^*}{2}.
\end{equation*}
If furthermore $f(t)$ is monotone, then we can take $U=V$.
\end{theorem}

A proof of Theorem 1.1 can be found in  \cite{B1} and \cite{BL}. If  $0\in \Omega$ and
$f(0)\ge0$, then Theorem \ref{T-1.1} holds also for sub-unital maps as
\eqref{Choi} and \eqref{HP} do so.

\section{Jensen inequalities for power means}

In this paper, $\bE$ denotes a unital (or sub-unital) positive linear map between two
matrix algebras $\bM_n$ and $\bM_m$. Here, $\bE$ is sub-unital if $\bE(I)\le I$, where
$I$ denotes the identity of any matrix algebra. We aim to extend the fundamental
inequality (1.1) to the maps on $\bM_n^+$ defined for $p\in(0,1]$ by
$$
\bE_p(Z):=\bE^{1/p}(Z^p).
$$
For the limit case $p=0$ one can define
\begin{equation*}
\bE_0(Z):=\lim_{p\searrow0}\bE_p(Z)=\exp\bE(\log Z)
\end{equation*}
as long as $\bE$ is unital and $Z\in\bM_n^+$ is invertible. 
Indeed, under these
assumptions, $\bE_p(Z)$ is also invertible and

\begin{align*}
\log\bE_p(Z)&=\frac{1}{p}\log\bE(I+p\log Z+o(p)) \\
&=\frac{1}{p}\log(I+p\bE(\log Z)+o(p))\longrightarrow\bE(\log Z)
\quad\mbox{as $p\searrow0$}.
\end{align*}

Thus, considering $\bE$ as a kind of arithmetic mean and \eqref{Choi} as the corresponding
Jensen inequality, we are looking for analogous Jensen type inequalities for the $p$th
power map $\bE_p$ with $p\in(0,1]$.
The assumption of operator concavity is not relevant to this purpose and inequalities for
the order relation in $\bM^+_m$ are not possible even for a function such as
$f(t)=\sqrt{t}$. However, with a reasonable concavity assumption, some meaningful
eigenvalue estimates hold. Our assumption is the doubly concavity of $f(t)$. We will say
that a function $f(t)$ is {\it doubly concave} if:
\begin{itemize}
\item[1.] $f(t)$ is a non-negative continuous function defined on a positive interval
$\Omega\subset[0,\infty)$,
\item[2.] $f(t)$ is  concave in the usual sense,
\item[3.] $f(t)$ is geometrically concave, i.e., $f(\sqrt{xy})\ge \sqrt{f(x)f(y)}$ for
all $x,y\in\Omega$.
\end{itemize}

\vskip 5pt\noindent
 If $f(t)$ and $g(t)$ are doubly concave on $\Omega$, then so is their geometric mean
$f^{\alpha}(t)g^{1-\alpha}(t)$ for $\alpha\in[0,1]$ and their minimum
$\min\{f(t),g(t)\}$. These properties with the following examples show that there are a
lot of doubly concave functions.

\begin{example}
Of course, the most important examples of doubly concave functions on
$\Omega=[0,\infty)$ are $t\mapsto t^q$ with exponent $q\in[0,1]$. Other simple examples
are $t\mapsto t/(t+1)$, $t\mapsto t/\sqrt{t+1}$ and $t\mapsto 1-e^{-t}$. However,
$\log(1+t)$ is not doubly concave on $[0,\infty)$.
\end{example}

\begin{example} 
On $\Omega=[1,\infty)$, the functions $\log t$ and $(t-1)^p$ for $p\in[0,1]$ are doubly
concave. For $q\ge 1$, the function $(t^q-1)^{1/q}$ is also doubly concave on
$[1,\infty)$.
\end{example}

\begin{example} 
On $\Omega=[0,1]$, the functions $ t(t-1)$ and $-t\log t$ are doubly concave, as well as
the function $\sqrt{1-t^2}$.
\end{example}

\begin{example}
The function $\sin t$ is doubly concave on $[0,\pi]$ and the function $\cos t$ is doubly
concave on $[0,\pi/2]$. More generally, for $\alpha,\beta \ge 0$ such that
$\alpha +\beta\le 1$, the function $\sin^{\alpha}t\cos^{\beta}t$ is doubly concave on 
$[0,\pi/2]$, as well as the function $\min\{\sin t, \cos t\}$.
\end{example}

\begin{example}
Let $\alpha>0$. The function $\alpha-|t-\alpha|$ is doubly concave on $[0,2\alpha]$.
More generally, let $0<\alpha_1\le \alpha_2<\beta$ and define a piecewise linear
function by $f(0)=f(\beta)=0$, $f(\alpha_1)=f(\alpha_2)>0$ and by the
condition that $f(t)$ is linear on each interval $[0,\alpha_1]$, $[\alpha_1,\alpha_2]$
and $[\alpha_2,\beta]$. Then $f(t)$ is doubly concave on $[0,\beta]$.
\end{example}

Our last example is of a rather general nature and is a straightforward consequence of
the arithmetic-geometric mean inequality.

\begin{example}
All non-negative, non-increasing, continuous concave functions defined on an interval
$[0,\beta]$ are doubly concave.
\end{example}

We have the following Jensen inequalities for power means associated to a unital
positive linear map $\bE:\bM_n\to\bM_m$. The unitality assumption can be relaxed to
sub-unitality. 

\begin{theorem}\label{T-2.7}
Let $\bE:\bM_n\to\bM_m$ be a sub-unital positive linear map. If $f(t)$ is a doubly
concave function on  $\Omega$, $Z\in\bM_n\{\Omega\}$, and $p\in(0,1]$, then
\begin{equation*}
f(\bE_p(Z)) \prec^{w(\log)} \bE_p(f(Z)).
\end{equation*}
If furthermore $f(t)$ is monotone, then, for some unitary $V\in \bM_m$,
\begin{equation*}
f(\bE_p(Z)) \ge V\bE_p(f(Z))V^*.
\end{equation*}

Moreover, the above assertions hold for $p=0$ too when $\bE$ is unital and both $Z$ and
$f(Z)$ are invertible.
\end{theorem}

If $\Omega$ is an unbounded interval of $[0,\infty)$, a non-negative concave function on
$\Omega$ is automatically non-decreasing, so that the second stronger estimate holds.
The following is the special case for the power means \eqref{p-means}. Note that the
$p=0$ case of \eqref{p-means} is
\begin{equation}\label{F-2.1}
A\,\beta_0\,B:=\lim_{p\searrow0}A\,\beta_p\,B=\exp\biggl(\frac{\log A+\log B}{2}\biggr)
\end{equation}
for invertible $A,B\in\bM_n^+$.

\vskip 5pt
\begin{cor}\label{C-2.8}
Let $f(t)$ be a doubly concave function on  $[0,\infty)$,  let $A,B\in\bM_n^+$ and
$0< p\le 1$. Then,
\begin{equation*}
f(A\,\beta_p\,B) \ge V\{f(A)\,\beta_p\,f(B)\}V^*
\end{equation*}
for some unitary $V\in \bM_n$. Moreover, this holds for $p=0$ too when $f(t)$ is not
identically zero and $A,B$ are invertible.
\end{cor}

\vskip 5pt
The corollary follows by applying Theorem \ref{T-2.7} to $Z=A\oplus B$  and
$\bE:\bM_{2n}\to\bM_n$, 
$$
\bE\biggl(\begin{bmatrix} A&X\\ Y&B\end{bmatrix}\biggr):=\frac{A+B}{2}.
$$
Note that, except the trivial case $f\equiv0$, $f(t)>0$ for all $t\in(0,\infty)$ and
hence $f(A)$ is invertible whenever so is $A$.

It is not possible to delete the unitary $V$ in Corollary \ref{C-2.8}, even for a doubly
concave and operator concave function. For instance, if $f(t)=t^{1/3}$ and $p=1/3$ then
we cannot have $(A\,\beta_{1/3}\,B)^{1/3}\ge A^{1/3}\,\beta_{1/3}\,B^{1/3}$, since it
would imply that $t\to t^3$ is operator convex, a contradiction.

Another special case of Theorem \ref{T-2.7} deals with the Schur product $X\circ Y$ of
$\bM_n$ (the entrywise product of $X$ and $Y$). This follows from the fact that
$Z\mapsto A\circ Z$ is a positive and sub-unital linear map  when $A\in\bM_n^+$ has
diagonal entries less than or equal to $1$.

\begin{cor}
Let $f(t)$ be a doubly concave function on  $[0,\infty)$, let $A,Z\in\bM_n^+$ and
$0< p\le 1$. Assume that the diagonal entries of $A$ are all less than or equal to $1$.
Then,
\begin{equation*}
f(\{A\circ Z^p\}^{1/p}) \ge V\{A\circ (f(Z))^p\}^{1/p}V^*
\end{equation*}
for some unitary $V\in \bM_n$.
\end{cor}

\vskip 5pt
We turn to the proof of the theorem. For this we first give a lemma.

\begin{lemma}\label{L-2.10}
If $f(t)$ is a doubly concave function on $\Omega$ and $p\in(0,1]$, then $f^p(t^{1/p})$
is concave on $\Omega^p:=\{t^p:t\in\Omega\}$.
\end{lemma}

\begin{proof}
We may assume that $\Omega$ is an open interval. Then we can further assume that $f(x)$
is strictly positive on $\Omega$; otherwise $f(x)$ must be identically zero. The
concavity of $f$ on $\Omega$ means that the right derivative $f_+'(x)$ is non-increasing
on $\Omega$. The geometric concavity of $f(x)$ is equivalent to the concavity of
$g(t):=\log f(e^t)$ on $\log\Omega:=\{\log x:x\in\Omega\}$. Notice that the
right derivative of $g(t)$ is $g_+'(t)=e^tf_+'(e^t)/f(e^t)$. In fact, this is seen by
taking the limit as $\delta\searrow0$ of
$$
\frac{g(t+\delta)-g(t)}{\delta}=\frac{e^{t+\delta}-e^t}{\delta}
\cdot\frac{f(e^{t+\delta})-f(e^t)}{e^{t+\delta}-e^t}
\cdot\frac{\log f(e^{t+\delta})-\log f(e^t)}{f(e^{t+\delta})-f(e^t)},
$$
where the above last term can be replaced with $1/f(e^t)$ if $f(e^{t+\delta})=f(e^t)$.
Hence it follows that $xf_+'(x)/f(x)$ is non-increasing on $\Omega$. Next, consider the
function $h(t):=f^p(t^{1/p})$ on $\Omega^p$. By a similar argument, we notice that the
right derivative of $h(t)$ is $h_+'(t)=t^{\frac{1}{p}-1}f^{p-1}(t^{1/p})f_+'(t^{1/p})$.
Thus, the concavity of $h(t)$ on $\Omega^p$ is equivalent to that
$x^{1-p}f^{p-1}(x)f_+'(x)$ is non-increasing on $\Omega$. Since
$$
x^{1-p}f^{p-1}(x)f_+'(x)=\begin{cases}
\{xf_+'(x)/f(x)\}^{1-p}\{f_+'(x)\}^p & \text{if $f_+'(x)\ge0$}, \\
-\{-xf_+'(x)/f(x)\}^{1-p}\{-f_+'(x)\}^p
& \text{if $f_+'(x)\le0$},
\end{cases}
$$
this indeed holds.
\end{proof}

\noindent
{\it Proof of Theorem \ref{T-2.7}.}\enspace
Assume that $0<p\le1$. For any $Z\in\bM_n\{\Omega\}$ let $X:=Z^p\in\bM_n\{\Omega^p\}$. By
Lemma \ref{L-2.10} we can apply Theorem \ref{T-1.1} to the function $f^p(t^{1/p})$ so that
we have 
\begin{equation*}
f^p(\bE^{1/p}(X)) \ge \frac{U\bE(f^p(X^{1/p}))U^*+V\bE(f^p(X^{1/p}))V^*}{2}
\end{equation*}
for some unitaries $U,\, V$. We thus obtain
\begin{equation}\label{F-2.2}
f^p(\bE^{1/p}(Z^p)) \ge \frac{U\bE(f^p(Z))U^*+V\bE(f^p(Z))V^*}{2},
\end{equation}
which yields the supermajorization
\begin{equation}\label{F-2.3}
f^p(\bE^{1/p}(Z^p)) \prec^w \bE(f^p(Z)).
\end{equation}
Here, we notice that if $A,B\in\bM_n^+$ and $A\prec^wB$, then
$A\prec^{w(\log)}B$. Indeed, to see this we may assume that $A,B$ are invertible.
The increasing convex function $-\log(-x)$ on $(-\infty,0)$ is applied to $-A\prec_w-B$
(which is equivalent to $A\prec^wB$) so that we have $-\log A\prec_w-\log B$. This means
that $A\prec^{w(\log)}B$.
Therefore, \eqref{F-2.3} entails the log-supermajorization
\begin{equation*}
f^p(\bE^{1/p}(Z^p)) \prec^{w(\log)} \bE(f^p(Z)),
\end{equation*}
which is equivalent to
\begin{equation*}
f(\bE^{1/p}(Z^p)) \prec^{w(\log)} \bE^{1/p}(f^p(Z)).
\end{equation*}
This proves the first assertion of the theorem. In case of an additional monotony assumption
on $f(t)$, we have $U=V$ in \eqref{F-2.2} so that
\begin{equation*}
f^p(\bE^{1/p}(Z^p)) \ge U\bE(f^p(Z))U^*.
\end{equation*}
Since $t\mapsto t^{1/p}$ is increasing, it follows that
\begin{equation*}
f(\bE^{1/p}(Z^p)) \ge V\bE^{1/p}(f^p(Z))V^*
\end{equation*}
for some unitary $V$. This proves the second assertion. The last assertion for the case
$p=0$ is immediately seen by taking the limit as $p\searrow0$ of the above estimates.\qed

\bigskip
As another consequence of Theorem \ref{T-2.7} (or Corollary \ref{C-2.8}) we have the
following determinantal inequality. The proof of a more general result will be given in
Section 4, Proposition \ref{P-4.12}. Note that $a\,\beta_0\,b$ is defined for all scalars
$a,b\ge0$ in such a way that $a\,\beta_0\,b=0$ if $a=0$ or $b=0$.

\begin{cor}\label{C-2.11}
Let $f(t)$ be a doubly concave function on  $\Omega$, let $A,B\in\bM_n\{\Omega\}$ and
$0< p\le 1$. Then
\begin{equation*}
{\det}^{1/n} f(A\,\beta_p\,B) \, \ge
\, \{{\det}^{1/n}f(A)\}\, \beta_p\, \{{\det}^{1/n} f(B) \}.
\end{equation*}
Moreover, this holds for $p=0$ too when $A,B$ are invertible.
\end{cor}

Corollary \ref{C-2.11} for $p=1$ and $f(t)=t$ is Minkowski's inequality.

Next, we may define doubly convex functions in a similar way. A function $g(t)$ is
{\it doubly convex} if:
\begin{itemize}
\item[1.] $g(t)$ is a non-negative continuous function defined on a positive interval
$\Omega\subset[0,\infty)$,
\item[2.] $g(t)$ is convex,
\item[3.] $g(t)$ is geometrically convex, i.e., $g(\sqrt{xy})\le \sqrt{g(x)g(y)}$ for all
$x,y\in\Omega$.
\end{itemize}

\begin{example} 
Given real numbers $c_i\ge 0$ and $\alpha_i\in (-\infty,0]\cup[1,\infty)$, $i=1,\dots,n$,
the function $g(t):= \sum_{i=1}^n c_i t^{\alpha_i}$ is doubly convex on $(0,\infty)$.
\end{example}

Double convexity will be used in Section 4. This notion is not relevant to the following
convex version of Theorem \ref{T-2.7}. It suffices to use merely convex functions, but a
monotony assumption is necessary. 

\begin{prop}\label{P-2.13}
Let $\bE:\bM_n\to\bM_m$ be a sub-unital positive linear map. If $g(t)$ is a non-negative
convex function on $[0,\infty)$ with $g(0)=0$, $Z\in\bM_n^+$, and $q\ge 1$, then, for some
unitary $V\in \bM_m$,
\begin{equation*}
g(\bE_q(Z)) \le V\bE_q(g(Z))V^*.
\end{equation*}
If $\bE$ is unital, then the above estimate holds also for any decreasing, non-negative
convex function on $(0,\infty)$ and any invertible $Z\in\bM_n^+$.
\end{prop}

We have statements, with reverse inequalities, similar to the previous corollaries for 
doubly concave functions. For instance:

\noindent
\begin{cor}\label{C-2.14}
Let $g(t)$ be a non-negative convex function on $[0,\infty)$ with $g(0)=0$,  
let $A,B\in\bM_n^+$ and $q\ge 1$. Then,
\begin{equation*}
g(A\,\beta_q\,B) \le V\{g(A)\,\beta_q\,g(B)\}V^*
\end{equation*}
for some unitary $V\in \bM_n$.
\end{cor}

We turn to the proof of the proposition.

\vskip 5pt\noindent
{\it Proof of Proposition \ref{P-2.13}.}\enspace
Considering $g(t)+\varepsilon t$ or $g(t)+\varepsilon$ for any $\varepsilon>0$, we can
assume that $g(t)>0$ for all $t>0$. Note that $g(t)$ is necessarily continuous, right
differentiable, and the right derivative of $h(t):=g^q(t^{1/q})$ on $(0,\infty)$ is
$h_+'(t)=t^{\frac{1}{q}-1}g^{q-1}(t^{1/q})g_+'(t^{1/q})$ as in the proof of Lemma
\ref{L-2.10}. Thus, the convexity of $h(t)$ on $(0,\infty)$ is equivalent to that
$\{g(x)/x\}^{q-1}g_+'(x)$ is non-decreasing on $(0,\infty)$. This indeed holds:
If $g(t)$ is convex with $g(0)=0$, then both $g(x)/x$ and $g_+'(x)$ are non-negative and
non-decreasing. On the other hand, if $g(t)$ is convex and decreasing, then $g(x)/x$ is
non-increasing and $g_+'(x)$ are non-decreasing with opposite signs. Therefore, under our
assumption, $g^q(t^{1/q})$ is convex. We may then apply the convex version of Theorem
\ref{T-1.1} and argue as in the proof of Theorem \ref{T-2.7}.\qed

\section{Minkowski type inequalities}

Section 3 is a bridge between Sections 2 and 4. Here, we will focus on Minkowski
determinantal type inequalities. Our setting is the theory of operator means in the
Kubo-Ando sense \cite{KA}, regarded as genuine non-commutative means. An important property
of operator means is the compatibility with congruence maps $A\mapsto X^*AX$, that is, for
every $A,B\in \bM_n^+$ and every invertible $X\in\bM_n$,
\begin{equation}\label{F-3.1}
(X^*AX)\,\sigma\,(X^*BX) = X^*(A\,\sigma\,B)X.
\end{equation}
From this and simultaneous diagonalization, we see that an operator mean is determined by
its value on commuting operators. The fact that invertibility of $X$ is crucial for
\eqref{F-3.1} should be stressed. For general $X$ we only have
$(X^*AX)\,\sigma\,(X^*BX) \ge X^*(A\,\sigma\,B)X$, called the transformer inequality,
and more generally for any positive linear map
$\bE : \bM_n \to \bM_m$,
\begin{equation*}
\bE(A) \,\sigma\, \bE(B) \ge \bE(A\,\sigma\,B).
\end{equation*}
This is essentially due to Ando \cite{An}, and it is related to the fact that $\sigma$ is
not necessarily continuous on the boundary of $\bM_n^+$, the non-invertible part of
$\bM_n^+$. We only have continuity from above; in particular,
\begin{equation}\label{F-3.2}
A\,\sigma\,B = \lim_{\varepsilon\searrow 0} (A+\varepsilon I)\,\sigma\,(B +\varepsilon I).
\end{equation}
Each operator mean $\sigma$ is associated with a non-negative operator monotone function
$h(t)$ on $[0,\infty)$ with $h(1)=1$, the representing function of $\sigma$. For every
invertible $A,B\in \bM_n^+$ we have
$$
A\,\sigma\,B = A^{1/2}h(A^{-1/2}BA^{-1/2})A^{1/2}.
$$
This is, together with \eqref{F-3.2}, the definition of $\sigma$ in terms of the function
$h(t)$. With a suitable assumption on the representing function, we obtain below some
Minkowski type majorizations.

 The famous Minkowski determinantal inequality is
\begin{equation}\label{F-3.3}
{\det}^{1/n}(A+B)\ge{\det}^{1/n}A+{\det}^{1/n}B
\end{equation}
for any  $A,B\in\bM_n^+$. In the rest of the paper, for any $X\in\bM_n$, we write
$\mu_1(X)\ge\dots\ge\mu_n(X)$ for the singular values of $X$ (i.e., the eigenvalues of
$|X|$) in decreasing order with multiplicities. In \cite{BH1} we noted that
\eqref{F-3.3} can be extended to
$$
\Biggl\{\prod_{j=n+1-k}^n\mu_j(A+B)\Biggr\}^{1/k}
\ge\Biggl\{\prod_{j=n+1-k}^n\mu_j(A)\Biggr\}^{1/k}
+\Biggl\{\prod_{j=n+1-k}^k\mu_j(B)\Biggr\}^{1/k}
$$
or equivalently,
\begin{equation}\label{F-3.4}
\Biggl\{\prod_{j=n+1-k}^n\mu_j(A\,\triangledown\,B)\Biggr\}^{1/k}
\ge\Biggl\{\prod_{j=n+1-k}^n\mu_j(A)\Biggr\}^{1/k}
\,\triangledown\,\Biggl\{\prod_{j=n+1-k}^k\mu_j(B)\Biggr\}^{1/k}
\end{equation}
for $ k=1,\dots,n$,
where $\triangledown$ stands for the arithmetic mean. Replace
$A,B$ with $(A+\varepsilon I)^{-1},(B+\varepsilon I)^{-1}$, respectively, in \eqref{F-3.4},
take the inverse of the both sides, and let $\varepsilon\searrow0$. Then we also have
\begin{equation}\label{F-3.5}
\Biggl\{\prod_{j=1}^k\mu_j(A\,!\,B)\Biggr\}^{1/k}
\le\Biggl\{\prod_{j=1}^k\mu_j(A)\Biggr\}^{1/k}
\,!\,\Biggl\{\prod_{j=1}^k\mu_j(B)\Biggr\}^{1/k},
\end{equation}
where $!$ stands for the harmonic mean, $A\,!\,B:=2(A^{-1}+B^{-1})^{-1}$.

In the next theorem we obtain majorizations similar to \eqref{F-3.4} and \eqref{F-3.5} for
more general operator means, but their forms are rather weaker than those of \eqref{F-3.4}
and \eqref{F-3.5}.

\vskip 10pt
\begin{theorem} Let $\sigma$  be an operator mean with the representing function $h(t)$.
\begin{itemize}
\item[\rm(i)] Assume that $h(t)$ is geometrically convex. Then, for every
$A,B\in\bM_n^+$ and $k=1,\dots,n$,
\begin{equation}\label{F-3.6}
\Biggl\{\prod_{j=1}^k\mu_j(A\,\sigma\,B)\Biggr\}^{1/k}
\ge\Biggl\{\prod_{j=1}^k\mu_j(A)\Biggr\}^{1/k}
\,\sigma\,\Biggl\{\prod_{j=n+1-k}^n\mu_j(B)\Biggr\}^{1/k},
\end{equation}
\begin{equation}\label{F-3.7}
\Biggl\{\prod_{j=1}^k\mu_j(A\,\sigma\,B)\Biggr\}^{1/k}
\ge\Biggl\{\prod_{j=n+1-k}^n\mu_j(A)\Biggr\}^{1/k}
\,\sigma\,\Biggl\{\prod_{j=1}^k\mu_j(B)\Biggr\}^{1/k}.
\end{equation}
\item[\rm(ii)]  Assume that $h(t)$ is geometrically concave. Then, for every
$A,B\in\bM_n^+$ and $k=1,\dots,n$,
\begin{equation}\label{F-3.8}
\Biggl\{\prod_{j=n+1-k}^n\mu_j(A\,\sigma\,B)\Biggr\}^{1/k}
\le\Biggl\{\prod_{j=n+1-k}^n\mu_j(A)\Biggr\}^{1/k}
\,\sigma\,\Biggl\{\prod_{j=1}^k\mu_j(B)\Biggr\}^{1/k},
\end{equation}
\begin{equation}\label{F-3.9}
\Biggl\{\prod_{j=n+1-k}^n\mu_j(A\,\sigma\,B)\Biggr\}^{1/k}
\le\Biggl\{\prod_{j=1}^k\mu_j(A)\Biggr\}^{1/k}
\,\sigma\,\Biggl\{\prod_{j=n+1-k}^n\mu_j(B)\Biggr\}^{1/k}.
\end{equation}
\end{itemize}
\end{theorem}

\vskip 10pt
\begin{proof}
(i)\enspace
To prove \eqref{F-3.6}, we may assume by continuity from above that $A$ and $B$ are
invertible. Hence $A\,\sigma\,B=A^{1/2}h(A^{-1/2}BA^{-1/2})A^{1/2}$, so we have
\begin{align*}
\Biggl\{\prod_{j=1}^k\mu_j(A\,\sigma\,B)\Biggr\}^{1/k} 
&=\Biggl\{\prod_{j=1}^k\mu_j(A^{1/2}h^{1/2}(A^{-1/2}BA^{-1/2}))\Biggr\}^{2/k} \\
\ge\Biggl\{\prod_{j=1}^k&\mu_j(A^{1/2})\Biggr\}^{2/k}
\Biggl\{\prod_{j=n+1-k}^n\mu_j(h^{1/2}(A^{-1/2}BA^{-1/2}))\Biggr\}^{2/k} \\
=\Biggl\{\prod_{j=1}^k&\mu_j(A)\Biggr\}^{1/k}
\Biggl\{\prod_{j=n+1-k}^nh(\mu_j(A^{-1/2}BA^{-1/2}))\Biggr\}^{1/k} \\
\ge\Biggl\{\prod_{j=1}^k&\mu_j(A)\Biggr\}^{1/k}
h\Biggl(\Biggl\{\prod_{j=n+1-k}^n\mu_j(A^{-1/2}B^{1/2})\Biggr\}^{2/k}\Biggr) \\
\ge\Biggl\{\prod_{j=1}^k&\mu_j(A)\Biggr\}^{1/k}
h\Biggl(\Biggl\{\prod_{j=n+1-k}^n\mu_j(A^{-1/2})\Biggr\}^{2/k}
\Biggl\{\prod_{j=n+1-k}^n\mu_j(B^{1/2})\Biggr\}^{2/k}\Biggr) \\
=\Biggl\{\prod_{j=1}^k&\mu_j(A)\Biggr\}^{1/k}
h\Biggl(\Biggl\{\prod_{j=1}^k\mu_j(A)\Biggr\}^{-1/k}
\Biggl\{\prod_{j=n+1-k}^n\mu_j(B)\Biggr\}^{1/k}\Biggr).
\end{align*}
In the above, the first inequality is due to the Gel'fand-Naimark majorization
(\cite[p.~248]{MO}, \cite[III.4.5]{Bh}), the second is due to the geometric
convexity of $h(t)$, and the last is due to the Horn majorization
(\cite[p.~246]{MO}, \cite[(III.19)]{Bh}). Hence \eqref{F-3.6} is obtained.
The proof of \eqref{F-3.7} is similar, or else we can show it from \eqref{F-3.6}
as follows: Consider the transposed operator mean $A\,\sigma'\,B:=B\,\sigma\,A$ with
the corresponding representing function $\tilde h(t):=t h(t^{-1})$ for $t>0$
(and $\tilde h(0):=\lim_{t\searrow0}\tilde h(t)$).
Since $\tilde h(t)$ is geometrically convex, we can apply \eqref{F-3.6} to $A$ and $B$
interchanged so that \eqref{F-3.7} follows.

(ii)\enspace
We may assume as above that $A,B$ are invertible. We can infer \eqref{F-3.8} from
\eqref{F-3.6}. Indeed, consider the adjoint operator mean
$A\,\sigma^*\,B:=(A^{-1}\,\sigma\,B^{-1})^{-1}$ for invertible $A,B$ with the representing
function $h^*(t):=h(t^{-1})^{-1}$ for $t>0$. Since $h^*(t)$ is geometrically concave if
and only if $h(t)$ is geometrically convex, we can apply \eqref{F-3.6} to $A^{-1}$ and
$B^{-1}$ to obtain
$$
\Biggl\{\prod_{j=1}^k\mu_j(A^{-1}\,\sigma^*\,B^{-1})\Biggr\}^{1/k}
\ge\Biggl\{\prod_{j=1}^k\mu_j(A^{-1})\Biggr\}^{1/k}
\,\sigma^*\,\Biggl\{\prod_{j=n+1-k}^n\mu_j(B^{-1})\Biggr\}^{1/k}.
$$
By reversing both sides we have \eqref{F-3.8}, which also implies \eqref{F-3.9} as in the
proof of (i).
\end{proof}

\vskip 10pt
\begin{cor}\label{C-3.2}
Let $\sigma$ be an operator mean whose representing function is geometrically convex.
Then, for every $A,B\in\bM_n^+$,
$$
{\det}^{1/n}(A\,\sigma\,B)\ge({\det}^{1/n}A)\,\sigma\,({\det}^{1/n}B),
$$
and the reverse inequality holds if the representing function is geometrically concave.
\end{cor}

\vskip 10pt
\begin{remark} It is obvious that
the majorization with $\prod_{j=1}^k$ (resp., $\prod_{j=n+1-k}^n$) in the all three terms
in \eqref{F-3.6} (resp., \eqref{F-3.8}) does not hold. In fact, for the arithmetic mean
(resp., the harmonic mean) and $k=1$, this means that $\mu_1(A+B)\ge\mu_1(A)+\mu_1(B)$
that is of course false.
\end{remark}

\begin{remark}
The arithmetic mean and the logarithmic operator mean (see Example \ref{E-3.12} below)
satisfy the assumption of (i), and the harmonic mean satisfies the assumption of (ii).
The geometric operator mean obviously satisfies both assumptions. We do not know whether,
under the assumption of (i), the generalization
$$
\Biggl\{\prod_{j=n+1-k}^n\mu_j(A\,\sigma\,B)\Biggr\}^{1/k}
\ge\Biggl\{\prod_{j=n+1-k}^n\mu_j(A)\Biggr\}^{1/k}
\,\sigma\,\Biggl\{\prod_{j=n+1-k}^n\mu_j(B)\Biggr\}^{1/k}
$$
of \eqref{F-3.4} and Corollary \ref{C-3.2} holds or not, and whether, under the
assumption of (ii), the generalization
$$
\Biggl\{\prod_{j=1}^k\mu_j(A\,\sigma\,B)\Biggr\}^{1/k}
\le\Biggl\{\prod_{j=1}^k\mu_j(A)\Biggr\}^{1/k}
\,\sigma\,\Biggl\{\prod_{j=1}^k\mu_j(B)\Biggr\}^{1/k}
$$
of \eqref{F-3.5} holds or not. But, these hold true for the weighted geometric operator
means as stated in the next proposition.
\end{remark}

For each $\alpha\in[0,1]$ let $A\,\#_\alpha\,B$ denote the $\alpha$-weighted geometric
mean of $A,B\in\bM_n^+$, defined for invertible $A,B$ as
$$
A\,\#_\alpha\,B:=A^{1/2}(A^{-1/2}BA^{-1/2})^\alpha A^{1/2}.
$$

\vskip 10pt
\begin{prop}
Let $\alpha\in[0,1]$. For every $A,B\in\bM_n^+$ and $k=1,\dots,n$,
\begin{align}
&\Biggl\{\prod_{j=1}^k\mu_j(A)\Biggr\}
\,\#_\alpha\,\Biggl\{\prod_{j=n+1-k}^n\mu_j(B)\Biggr\} \nonumber\\
&\qquad\le\prod_{j=1}^k\mu_j(A\,\#_\alpha\,B)
\le\Biggl\{\prod_{j=1}^k\mu_j(A)\Biggr\}\,\#_\alpha\,
\Biggl\{\prod_{j=1}^k\mu_j(B)\Biggr\}, \label{F-3.10}
\end{align}
\begin{align}
&\Biggl\{\prod_{j=n+1-k}^n\mu_j(A)\Biggr\}
\,\#_\alpha\,\Biggl\{\prod_{j=1}^k\mu_j(B)\Biggr\} \nonumber\\
&\qquad\ge\prod_{j=n+1-k}^n\mu_j(A\,\#_\alpha\,B)
\ge\Biggl\{\prod_{j=n+1-k}^n\mu_j(A)\Biggr\}\,\#_\alpha\,
\Biggl\{\prod_{j=n+1-k}^n\mu_j(B)\Biggr\}. \label{F-3.11}
\end{align}
\end{prop}

\vskip 10pt
\begin{proof}
Since the representing function of $\#_\alpha$ is $t^{\alpha}$, the first inequalities of
\eqref{F-3.10} and \eqref{F-3.11} are special cases of \eqref{F-3.6} and \eqref{F-3.8},
respectively. Let us prove the second inequality of \eqref{F-3.10}. We may assume that
$A,B$ are invertible and $0<\alpha<1$. Since
$A^{\frac{1-\alpha}{2\alpha}}BA^{\frac{1-\alpha}{2\alpha}}\le I$ implies
$A\,\#_\alpha\,B\le I$ as easily verified, we have
$\|A\,\#_\alpha\,B\|_\infty\le\big\|\bigl(A^{\frac{1-\alpha}{2\alpha}}B
A^{\frac{1-\alpha}{2\alpha}}\bigr)^\alpha\big\|_\infty$ for the operator norm. With the
antisymmetric tensor power technique (see \cite{AH}, \cite[Section 4.6]{H1}) this yields
$$
\prod_{j=1}^k\mu_j(A\,\#_\alpha\,B)\le\prod_{j=1}^k\mu_j\bigl(
\bigl(A^{\frac{1-\alpha}{2\alpha}}BA^{\frac{1-\alpha}{2\alpha}}\bigr)^\alpha\bigr),
\qquad k=1,\dots,n.
$$
Moreover, for $k=1,\dots,n$,
$$
\prod_{j=1}^k\mu_j\bigl(
\bigl(A^{\frac{1-\alpha}{2\alpha}}BA^{\frac{1-\alpha}{2\alpha}}\bigr)^\alpha\bigr)
=\Biggl\{\prod_{j=1}^k\mu_j\bigl(A^{\frac{1-\alpha}{2\alpha}}
B^{1/2}\bigr)\Biggr\}^{2\alpha}
\le\Biggl\{\prod_{j=1}^k\mu_j(A)\Biggr\}^{1-\alpha}
\Biggl\{\prod_{j=1}^k\mu_j(B)\Biggr\}^\alpha
$$
by the Horn majorization. Hence the second inequality of \eqref{F-3.10} follows. The second
inequality of \eqref{F-3.11} then follows from that of \eqref{F-3.10} by replacing $A,B$
with $A^{-1},B^{-1}$ and reversing the inequality.
\end{proof}

The following is a restatement of the second inequality of \eqref{F-3.10} or \eqref{F-3.11}
in terms of log-majorization, see \cite{AH}.

\begin{cor}\label{C-3.6}
For every $A,B\in\bM_n^+$ and every $\alpha\in[0,1]$,
\begin{equation}\label{F-3.12}
A\,\#_\alpha\,B\prec_{(\log)}A^\downarrow\,\#_\alpha\,B^\downarrow.
\end{equation}
\end{cor}

\vskip 10pt
\begin{proof}
The second inequality of \eqref{F-3.10} means that
$$
A\,\#_\alpha\,B\prec_{w(\log)}A^\downarrow\,\#_\alpha\,B^\downarrow.
$$
Since
$$
\prod_{j=1}^n\mu_j(A\,\#_\alpha\,B)=\det(A\,\#_\alpha\,B)
=(\det A)^{1-\alpha}(\det B)^\alpha
=\prod_{j=1}^n\bigl\{\mu_j^{1-\alpha}(A)\mu_j^\alpha(B)\bigr\},
$$
we have \eqref{F-3.12}.
\end{proof}

\vskip 10pt
\begin{prop}\label{P-3.7}
Let $\sigma$ be an operator mean with representing function $h(t)$. Assume that there
exists a probability measure $\nu$ on $[0,1]$ such that
\begin{equation}\label{F-3.13}
h(x)=\int_0^1x^\alpha\,d\nu(\alpha),\qquad x\in[0,\infty).
\end{equation}
Then, for every $A,B\in\bM_n^+$ and $k=1,\dots,n$,
\begin{equation}\label{F-3.14}
\Biggl\{\prod_{j=n+1-k}^n\mu_j(A\,\sigma\,B)\Biggr\}^{1/k}
\ge\Biggl\{\prod_{j=n+1-k}^n\mu_j(A)\Biggr\}^{1/k}\,\sigma\,
\Biggl\{\prod_{j=n+1-k}^n\mu_j(B)\Biggr\}^{1/k}.
\end{equation}
\end{prop}

\vskip 15pt
\begin{proof}
By assumption the operator mean $A\,\sigma\,B$ is expressed as
$$
A\,\sigma\,B=\int_0^1A\,\#_\alpha\,B\,d\nu(\alpha).
$$
Since $A\in\bM_n^+\mapsto\bigl\{\prod_{j=n+1-k}^n\mu_j(A)\bigr\}^{1/k}$ is superadditive
(hence concave) by \cite[Example 3.8]{BH1} (or Example \ref{E-4.5} below), 
\begin{align*}
\Biggl\{\prod_{j=n+1-k}^n\mu_j(A\,\sigma\,B)\Biggr\}^{1/k}
&\ge\int_0^1\Biggl\{\prod_{j=n+1-k}^n\mu_j(A\,\#_\alpha\,B)\Biggr\}^{1/k}\,d\nu(\alpha) \\
&\ge\int_0^1\Biggl\{\prod_{j=n+1-k}^n\mu_j(A)\Biggr\}^{1/k}\,\#_\alpha\,
\Biggl\{\prod_{j=n+1-k}^n\mu_j(B)\Biggr\}^{1/k}\,d\nu(\alpha) \\
&=\Biggl\{\prod_{j=n+1-k}^n\mu_j(A)\Biggr\}^{1/k}\,\sigma\,
\Biggl\{\prod_{j=n+1-k}^n\mu_j(B)\Biggr\}^{1/k},
\end{align*}
where \eqref{F-3.11} has been used for the second inequality.
\end{proof}

\vskip 15pt
\begin{remark}
As mentioned in the Introduction, we are interested in operator means $\sigma$ for which
the log-supermajorizations in \eqref{Conj} hold. For example, the second inequality
\eqref{F-3.11} (or the log-majorization of Corollary \ref{C-3.6}) is the second relation
of \eqref{Conj} for $\sigma=\#_\alpha$ and the first of \eqref{F-3.11} is slightly weaker
than that of \eqref{Conj} for $\sigma=\#_\alpha$. When $\sigma$ has the representing
function as in Proposition \ref{P-3.7}, \eqref{F-3.14} is a weaker version of the second
of \eqref{Conj} since $(a\,\sigma\,b)^{1/k}\ge a^{1/k}\,\sigma\,b^{1/k}$ for reals
$a,b\ge0$ in this case. Note also that sums of geometrically convex functions are again
such functions; so the function $h(t)$ given in \eqref{F-3.13} is geometrically convex.

When $\sigma=\triangledown$ the arithmetic mean, both relations of \eqref{Conj} hold
since \eqref{Fan-Lidskii} entails
$A^\downarrow+B^\uparrow\prec^{w(\log)}A+B\prec^{w(\log)}A^\downarrow+B^\downarrow$.
Furthermore, by replacing $A,B$ with $A^{-1},B^{-1}$ and reversing
the inequalities, we see that the relations in \eqref{Conj} hold with log-submajorization
$\prec_{w(\log)}$ instead of $\prec^{w(\log)}$ when $\sigma=\,!$ the harmonic mean.
When $\sigma=\#$ the geometric mean, both of \eqref{Conj} do hold; in fact,
the log-majorizations
$A^\downarrow \# B^\uparrow \prec^{(\log)}A\# B\prec^{(\log)}A^\downarrow\#B^\downarrow$
hold. This follows from the Gel'fand-Naimark and the Horn majorizations applied to the
factorization $A\#B =A^{1/2}VB^{1/2}$ where $V:=(A^{-1/2}BA^{-1/2})^{1/2}A^{1/2}B^{-1/2}$
is a unitary. On the other hand, if $AB=BA$, then both of \eqref{Conj} do hold also
when $\sigma=\sigma_p$ is the operator $p$-mean (see Example \ref{E-3.11}) for
$p\in (0,1)$.

After these considerations {\it we may conjecture \eqref{Conj} for any operator mean
$\sigma$ whose representing function is geometrically convex.}
\end{remark}

\vskip 10pt
In the rest of the section we will present a characterization of operator monotone
functions on $[0,\infty)$ admitting the integral expression \eqref{F-3.13} and give
concrete examples of such functions.

\vskip 5pt
\begin{theorem}\label{T-3.9}
The following conditions for a non-negative operator monotone function $h(x)$ on
$[0,\infty)$ are equivalent:
\begin{itemize}
\item[\rm(i)] there exists a finite positive measure $\nu$ on $[0,1]$ such that
$$
f(x)=\int_0^1x^\lambda\,d\nu(\lambda),\qquad x\in[0,\infty);
$$
\item[\rm(ii)] $h(e^t)$ is absolutely monotone on $\bR$, i.e.,
$\frac{d^n}{dt^n}\,h(e^t)\ge0$ for every $t\in\bR$ and $n\in\bN$, or equivalently,
$h(e^{-t})$ is completely monotone on $\bR$.
\end{itemize}
\end{theorem}

\begin{proof}
(i) $\Rightarrow$ (ii).\enspace
Assumption (i) means that
$$
h(e^{-t})=\int_0^1e^{-\lambda t}\,d\nu(\lambda),\qquad t\in\bR.
$$
Since $\frac{d^n}{dt^n}\,d^{-\lambda t}=(-\lambda)^ne^{-\lambda t}$ for every $t\in\bR$, we
have
$$
(-1)^n\frac{d^n}{dt^n}\,h(e^{-t})=\int_0^1\lambda^ne^{-\lambda t}\,d\nu(\lambda)\ge0,
\qquad t\in\bR,
$$
and so (ii) follows.

(ii) $\Rightarrow$ (i).\enspace
For each $\alpha\in\bR$, (ii) implies that $f(e^{-(t-\alpha)})$ is completely monotone on
$[0,\infty)$. Hence by Bernstein's representation theorem \cite{Wi} there exists a unique
positive finite measure $\nu_\alpha$ on $[0,\infty)$ such that
$$
h(e^{-(t-\alpha)})=\int_0^\infty e^{-\lambda t}\,d\nu_\alpha(\lambda),
\qquad t\in[0,\infty).
$$
Whenever $\alpha\ge0$, by replacing $t\ge0$ with $t+\alpha\ge0$ we have
$$
h(e^{-t})=\int_0^\infty e^{-\lambda t}e^{-\alpha\lambda}\,d\nu_\alpha(\lambda),
\qquad t\in[0,\infty).
$$
Thanks to the uniqueness of the representation measure in Bernstein's representation, we
have $d\nu_0(\lambda)=e^{-\alpha\lambda}\,d\nu_\alpha(\lambda)$ so that
$d\nu_\alpha(\lambda)=e^{\alpha\lambda}\,d\nu_0(\lambda)$ on $[0,\infty)$. Therefore,
$$
h(e^{-(t-\alpha)})=\int_0^\infty e^{-\lambda(t-\alpha)}\,d\nu_0(\lambda)
$$
for every $t\ge0$ and every $\alpha\ge0$, which implies that
$$
h(e^{-t})=\int_0^\infty e^{-\lambda t}\,d\nu_0(\lambda),
\qquad t\in\bR,
$$
that is,
\begin{equation}\label{F-3.15}
h(x)=\int_0^\infty x^\lambda\,d\nu_0(\lambda),
\qquad x\in(0,\infty).
\end{equation}
Now suppose that $\nu_0((1,\infty))>0$. Then we have
$$
\frac{h(x)}{x}\ge\int_{(1,\infty)}x^{\lambda-1}\,d\nu_0(\lambda)
\nearrow+\infty\ \ \mbox{as $x\nearrow\infty$},
$$
which contradicts the fact that $\lim_{x\to\infty}h(x)/x<+\infty$, easily verified from
the integral expression of $h(x)$ \cite[(V.53)]{Bh}. Hence $\nu_0((1,\infty))=0$ and
\eqref{F-3.15} is the required integral expression in (i). The equality for $x=0$ also
follows by taking the limit of \eqref{F-3.15} as $x\searrow0$.
\end{proof}

In the following let us consider three typical families of operator monotone functions
discussed in \cite{HK0,HK}. Examples show that operator monotone functions having the integral
expression \eqref{F-3.13} are not many.

\begin{example}\rm
For each $\alpha\in[0,1]$ the function
$$
h_\alpha(x):=\frac{x^\alpha+x^{1-\alpha}}{2}
$$
is an operator monotone function on $[0,\infty)$. It is clearly a special form of
\eqref{F-3.13}.
\end{example}

\begin{example}\label{E-3.11}\rm
For each $p\in[-1,1]$ the function
$$
b_p(x):=\biggl(\frac{x^p+1}{2}\biggr)^{1/p},
$$
where $b_0(x):=\lim_{p\to0}b_p(x)=\sqrt x$, is the representing function of
the operator $p$-mean $\sigma_p$ such that $A\,\sigma_p\,B = A\,\beta_p\,B$ when $AB=BA$. 
The function $b_p(x)$ is geometrically convex if $0\le p\le 1$ and geometrically concave
if $-1\le p\le 0$. On the other hand, when $p=1/m$ with $m\in\bN$, since
$$
b_{1/m}(x)=\biggl(\frac{x^{1/m}+1}{2}\biggr)^m
=\frac{1}{2^m}\sum_{k=0}^m{m\choose k}x^{k/m},
$$
$b_{1/m}$ has an integral form \eqref{F-3.13}. 
Now, suppose that $p\in(0,1)$ and $b_p$ is represented as in \eqref{F-3.13}. By
Theorem \ref{T-3.9}, $(e^t+1)^q$ must be absolutely monotone on $\bR$, where
$q:=p^{-1}$. Then $(e^t+1)^q$ can extend to an entire function, see \cite{Wi}.
But this is not the case unless $q$ is a positive integer, because $(e^z+1)^q$ has
a singularity at $z=i\pi$ for a non-integer $q$. Thus, for $p\in(0,1)$
such that $p^{-1}\notin \bN$, $b_p(x)$ does not admit the expression \eqref{F-3.13}.
\end{example}

\begin{example}\label{E-3.12}\rm
For each $\alpha\in[-1,2]$ the function
$$
f_\alpha(x):=\frac{\alpha-1}{\alpha}\cdot\frac{x^\alpha-1}{x^{\alpha-1}-1}
$$
is operator monotone  on $[0,\infty)$. Here, $f_{1/2}(x)=\sqrt x$ and
$f_1(x)=(x-1)/\log x$, the representing function of the logarithmic operator mean.
When $1<\alpha\le2$, we have
$$
\frac{d^2}{dt^2}\log f_\alpha(e^t)
=\frac{e^{(2\alpha-1)t}}{(e^{\alpha t}-1)^2(e^{(\alpha-1)t}-1)^2}\,\ffi(t),
$$
where
$$
\ffi(t):=(\alpha-1)^2e^{\alpha t}-\alpha^2e^{(\alpha-1)t}+2(2\alpha-1)
-\alpha^2e^{-(\alpha-1)t}+(\alpha-1)^2e^{-\alpha t}.
$$
Since
$$
\ffi''(t)=\alpha^2(\alpha-1)^2\bigl\{e^{\alpha t}-e^{(\alpha-1)t}
-e^{-(\alpha-1)t}+e^{-\alpha t}\bigr\}\ge0,
$$
we see that $\ffi(t)\ge0$ and hence $\frac{d^2}{dt^2}\log f_\alpha(e^t)\ge0$. When
$-1\le\alpha<1$, we have
$$
\log f_\alpha(e^t)=\log\frac{1-\alpha}{\alpha}
+\log\frac{e^{\alpha t}-1}{1-e^{(\alpha-1)t}},
$$
and $\frac{d^2}{dt^2}\log f_\alpha(e^t)$ is given in
the same expression as above with the same function $\ffi(t)$. If $1/2\le\alpha<1$, then
$\ffi''(t)\ge0$ so that $\frac{d^2}{dt^2}\log f_\alpha(e^t)\ge0$. If $-1\le\alpha<1/2$,
then $\ffi''(t)<0$ for $t\ne0$ so that $\frac{d^2}{dt^2}\log f_\alpha(e^t)<0$ for all
$t\ne0$. Therefore, $ f_\alpha(x)$ is geometrically  convex  for $\alpha\in[1/2,2]$ and
geometrically concave for $\alpha\in[-1,1/2)$.

Now, suppose that $\alpha\in[1/2,2]\setminus\{1\}$ and $f_\alpha(x)$ is represented as in
\eqref{F-3.13}. Then by Theorem \ref{T-3.9},
$$
\frac{\alpha-1}{\alpha}\cdot\frac{e^{\alpha t}-1}{e^{(\alpha-1)t}-1}
=\frac{\alpha-1}{\alpha}\cdot
\frac{\sinh\frac{\alpha t}{2}}{\sinh\frac{(\alpha-1)t}{2}}\,e^{t/2}
$$
is absolutely monotone on $\bR$ so that $\sinh\alpha t/\sinh(\alpha-1)t$ can extend to an
entire function. Since $\sinh(\alpha-1)z=0$ at $z=i\pi/(\alpha-1)$, we must have
$\sinh i\alpha\pi/(\alpha-1)=0$ so that $\alpha/(\alpha-1)=m\in\bZ$. Hence
$\alpha=m/(m-1)$ with $m\in\bZ\setminus\{0,1\}$. When $\alpha=m/(m-1)$ with
$m\in\{2,3,\dots\}$, we have
$$
f_\alpha(x)=\frac{1}{m}\sum_{k=0}^{m-1}x^{\frac{kt}{m-1}},
$$
which is a special case of \eqref{F-3.13}. When $\alpha=-m/(-m-1)=m/(m+1)$ with
$m\in\{1,2,\dots\}$, we have
$$
f_\alpha(x)=\frac{1}{m}\sum_{k=1}^mx^{\frac{kt}{m+1}},
$$
which is also a particular form of \eqref{F-3.13}. Therefore, $f_\alpha$ admits the
expression \eqref{F-3.13} if and only if
$$
\alpha\in\biggl\{\frac{m}{m+1}:m=1,2,\dots\biggr\}
\cup\{1\}\cup\biggl\{\frac{m+1}{m}:m=1,2,\dots\biggr\}.
$$
Note that $f_{m/(m+1)}(x)\nearrow f_1(x)$ and $f_{(m+1)/m}(x)\searrow f_1(x)$ as
$m\to\infty$.
\end{example}

\vskip 5pt
In this section, we have been concerned with operator means $A\,\sigma\,B$ whose
representing functions $h(x)$ is such that $h(e^t)$ is absolutely monotone on $\bR$.
Equivalently, these operator means are averages of weighted geometric means
$A\,\#_\alpha\,B$ expressed as
\begin{equation}\label{F-3.16}
A\,\sigma\,B=\int_0^1A\,\#_\alpha\,B\,d\nu(\alpha)
\end{equation}
for some probability measure $\nu$ on $[0,1]$. Since the path
$\alpha\in[0,1]\mapsto A\,\#_\alpha\,B$ is the geodesic from $A$ to $B$ for a natural
Riemannian metric on the positive definite matrices (see \cite{Bh2} and also Section 5
below), we call such an operator mean a {\it geodesic mean}. The next section considerably
extends Proposition \ref{P-3.7}. Indeed, an inequality more general than \eqref{F-3.14}
will be given in Corollary \ref{C-4.8} below. But we gave a brief independent proof of
Proposition \ref{P-3.7} to make this section self-contained.

\section{Anti-norms and operator means}

A symmetric norm $\|\cdot\|$, i.e., a unitarily invariant norm on $\bM_n$, can be defined
by its restriction to the positive cone $\bM_n^+$. Symmetric norms restricted on $\bM_n^+$
are characterized by the following three properties:
(i) $\| \lambda A\| = \lambda\|A\|$  for all $A\in \bM_n^+ $ and all reals $\lambda\ge 0$,
(ii) $\| UAU^*\|$ for all $A\in \bM_n^+$ and all unitaries $U\in\bM_n$, and
(iii) $\|A\| \le \|  A+B\| \le   \|  A\|  + \|  B\|$ for all $A,\,B\in \bM_n^+ $.

This section continues the study of geodesic means defined by \eqref{F-3.16}. We will
extend Proposition \ref{P-3.7} and obtain a Jensen/Minkowski inequality for quite a large
class of functionals that we call anti-norms as those are similar to symmetric norms but
with a reverse inequality.

\begin{definition}\label{D-4.1}
A {\it symmetric anti-norm} $\|\cdot\|_!$ on $\bM_n^+$ is a non-negative continuous
functional such that
\begin{itemize}
\item[1.] $\| \lambda A\|_! = \lambda\|A\|_!$  for all $A\in \bM_n^+ $ and all reals
$\lambda\ge 0$,
\item[2.] $\|  A\|_! = \| UAU^*\|_!$  for all $A\in \bM_n^+ $ and all unitaries $U$,
\item[3.] $\|  A+B\|_! \ge   \|  A\|_!  + \|  B\|_!$ for all $A,\,B\in \bM_n^+ $.
\end{itemize}
If further $\|A\|_!=0$ entails $A=0$, then we say that the anti-norm $\|\cdot\|_!$ is
{\it regular}.
\end{definition}

This definition without the continuity assumption was first introduced in \cite{BH1}. The
continuity assumption is not essential,  but deleting it would lead to rather strange
concave functionals, not continuous on the boundary of $\bM_n^+$ such as
$\| A\|_!:={\mathrm{Tr\,}}A$ if $A$ is invertible and $\| A\|_!:=0$ if $A$ is not
invertible.

The next two examples come from \cite{BH1}.

\begin{example}\label{E-4.1}
The trace norm is  an anti-norm! More generally for $k=1,\dots,n$, we define the
{\it Ky Fan $k$-anti-norm} on $\bM_n^+$ as the sum of the $k$ smallest eigenvalues, i.e.,
$$
\|A\|_{\{ k\}} :=\sum_{j=1}^k\mu_{n+1-j}(A),
$$
where $\mu_1(A)\ge\dots\ge\mu_n(A)$ are as before the eigenvalues of $A$ in decreasing
order. The anti-norm $\|\cdot\|_{\{k\}}$ is not regular except for $k=n$ (the trace norm).
\end{example}

\begin{example}\label{E-4.3}
For $-p<0$ the negative Schatten anti-norms are
$$
\Vert A\Vert_{-p} := \left\{\sum_{j=1}^n \mu_j^{-p}(A)\right\}^{-1/p}.
$$
That $A\mapsto \|A\|_{-p}$ is a superadditive functional on $\bM_n^+$ was noticed in
\cite{BH1}.
\end{example}

\begin{example}\label{E-4.4}
More generally, for $-p<0$ and $k= 1,\dots,n$, the negative Schatten-Ky Fan anti-norms
are
$$
\Vert A\Vert_{-p,k} := \left\{\sum_{j=1}^{k} \mu_{n+1-j}^{-p}(A)\right\}^{-1/p}.
$$
By definition note that $\|A\|_{-p,k}=0$ unless $A$ is invertible.
That the Schatten-Ky Fan anti-norms  $A\mapsto \|A\|_{-p,k}$ are superadditive on 
$\bM_n^+$ is a special case of Proposition \ref{P-4.6} below.
\end{example}

\begin{example}\label{E-4.5}
For $k=1,\dots,n$ the functional
$$
\Delta_k(A):=\left\{\prod_{j=1}^k \mu_{n+1-j}(A)\right\}^{1/k}
$$
is a symmetric anti-norm on $\bM_n^+$. Note that
\begin{equation}\label{F-4.1}
\Delta_k(A)
=\lim_{p\searrow0}\Biggl\{{1\over k}\sum_{j=1}^k\mu_{n+1-j}^{-p}(A)\Biggr\}^{-1/p}.
\end{equation}
\end{example}

\vskip 5pt
These examples illustrate the following general fact. 

\begin{prop}\label{P-4.6}
Let $\|\cdot\|$ be a symmetric norm on $\bM_n$ and $p>0$.  For  $A\in\bM_n^+$ set
$$
\|A\|_!:=\begin{cases}
\|A^{-p}\|^{-1/p} & \text{if $A$ is invertible},\\
0 & \text{otherwise}.
\end{cases}
$$
Then $\|\cdot\|_!$ is a symmetric anti-norm. 
\end{prop}

A symmetric anti-norm $\|\cdot\|_!$ occurring as above is called a {\it derived anti-norm}.

\begin{proof}
Let us first show the continuity of $\|\cdot\|_!$. It suffices to check that if $\{A_l\}$
is a sequence of invertible matrices in $\bM_n^+$ converging to a non-invertible
$A\in\bM_n^+$, then $\|A_l\|_!\to0$. For such $\{A_l\}$, since
$\varepsilon_l:=\mu_n(A_l)\to\mu_n(A)=0$, we have
$A_l^{-p}\ge\varepsilon_l^{-p}P_l$ and
so $\|A_l^{-p}\|\ge\varepsilon_l^{-p}\|P_l\|$, where $P_l$ is a rank one projection onto
an eigenvector of $A_l$ corresponding to $\varepsilon_l$. Hence
$\|A_l\|_!\le\varepsilon_l\|P_l\|^{-1/p}\to0$ since $\|P_l\|$ is a positive constant.

Let $\Phi$ be the symmetric gauge function corresponding to $\|\cdot\|$. Define for
$a\in\bR^n_+$,
$$
\Phi_!(a):=\begin{cases}
\Phi(a^{-p})^{-1/p} & \text{if $a\in(0,\infty)^n$}, \\
0 & \text{otherwise}.
\end{cases}
$$
We will show that $\Phi_!$ is superadditive on $\bR^n_+$. Then $\Phi_!$ is a symmetric
anti-gauge function since it is clearly permutation-invariant and homogeneous.
Since we of course have $\|A\|_!=\Phi_!(\mu(A))$ for all $A\in\bM_n^+$, it follows from
\cite[Proposition 3.2]{BH1} that $\|\cdot\|_!$ is a symmetric anti-norm.

Let $\Phi'$ be the symmetric gauge function dual to $\Phi$, see \cite[(4.4.4)]{H1}. For any
$a\in(0,\infty)^n$ we have
$$
\Phi(a^{-p})=\sup\Biggl\{\sum_{i=1}^na_i^{-p}b_i:
b=(b_i)\in[0,\infty)^n,\,\Phi'(b)=1\Biggr\}
$$
so that
\begin{equation}\label{F-4.2}
\Phi(a^{-p})^{-1/p}=\inf\Biggl\{\Biggl(\sum_{i=1}^na_i^{-p}b_i\Biggr)^{-1/p}:
b=(b_i)\in[0,\infty)^n,\,\Phi'(b)=1\Biggr\}.
\end{equation}
Let $a\in(0,\infty)^n$, $b\in[0,\infty)^n$, and $x=(x_i)\in\bR^n$. For every $t\in\bR$
such that $a+tx=(a_i+tx_i)\in(0,\infty)^n$ we compute
$$
{d\over dt}\Biggl(\sum_i(a_i+tx_i)^{-p}b_i\Biggr)^{-{1\over p}}
=\Biggl(\sum_i(a_i+tx_i)^{-p}b_i\Biggr)^{-{1\over p}-1}
\Biggl(\sum_i(a_i+tx_i)^{-p-1}x_ib_i\Biggr)
$$
and
\begin{align*}
&{d^2\over dt^2}\Biggl(\sum_i(a_i+tx_i)^{-p}b_i\Biggr)^{-{1\over p}} \\
&\quad=(p+1)\Biggl(\sum_i(a_i+tx_i)^{-p}b_i\Biggr)^{-{1\over p}-2}
\Biggl(\sum_i(a_i+tx_i)^{-p-1}x_ib_i\Biggr)^2 \\
&\quad\quad+(-p-1)\Biggl(\sum_i(a_i+tx_i)^{-p}b_i\Biggr)^{-{1\over p}-1}
\Biggl(\sum_i(a_i+tx_i)^{-p-2}x_i^2b_i\Biggr) \\
&\quad=(p+1)\Biggl(\sum_i(a_i+tx_i)^{-p}b_i\Biggr)^{-{1\over p}-2} \\
&\quad\quad\times\Biggl\{\Biggl(\sum_i(a_i+tx_i)^{-p-1}x_ib_i\Biggr)^2
-\Biggl(\sum_i(a_i+tx_i)^{-p}b_i\Biggr)
\Biggl(\sum_i(a_i+tx_i)^{-p-2}x_i^2b_i\Biggr)\Biggr\} \\
&\quad\le0
\end{align*}
thanks to the Schwarz inequality. Therefore,
$a\in(0,\infty)^n\mapsto\bigl(\sum_{i=1}^na_i^{-p}b_i\bigr)^{-1/p}$ is concave and so
superadditive due to positive homogeneity. Hence for $a,\tilde a\in(0,\infty)^n$ and
$b\in[0,\infty)^n$ with $\Phi'(b)=1$ we have
\begin{align*}
\Biggl(\sum_{i=1}^n(a_i+\tilde a_i)^{-p}b_i\Biggr)^{-1/p}
&\ge\Biggl(\sum_{i=1}^na_i^{-p}b_i\Biggr)^{-1/p}
+\Biggl(\sum_{i=1}^n\tilde a_i^{-p}b_i\Biggr)^{-1/p} \\
&\ge\Phi(a^{-p})^{-1/p}+\Phi(\tilde a^{-p})^{-1/p}.
\end{align*}
Taking the infimum of the left-hand side over $b$ as in \eqref{F-4.2} gives the required
superadditivity of $\Phi_!$.
\end{proof}

\vskip 10pt
\begin{theorem}\label{T-4.7}
If $f(t)$ is a doubly concave function on an interval $\Omega\subset[0,\infty)$ and
$A,B\in\bM_n\{\Omega\}$, then
\begin{equation*}
\| f(A\,\sigma\,B) \|_! \ge \|f(A)\|_! \,\sigma\, \| f(B) \|_!
\end{equation*}
 for all derived anti-norms $\|\cdot\|_!$  and all geodesic means $\sigma$.
\end{theorem}

\vskip 10pt
Applying the theorem to the anti-norms of Example \ref{E-4.4}, letting $p\searrow 0$ and
using \eqref{F-4.1} we obtain a generalization of Proposition \ref{P-3.7}.

\vskip 10pt
\begin{cor}\label{C-4.8}
If $f(t)$ is a doubly concave function on $\Omega$ and $A,B\in\bM_n\{\Omega\}$, then
\begin{equation*}
\Delta_k (f(A\,\sigma\, B)) \ge
\Delta_k(f(A))\, \sigma \,\Delta_k(f(B))
\end{equation*}
for all $k=1,\cdots,n$ and all geodesic means $\sigma$.
\end{cor}

\begin{remark}
Theorem \ref{T-4.7} holds true for all derived anti-norms. But it does not hold for any
regular anti-norm. In fact, let $\|\cdot\|_!$ be a regular anti-norm and $\sigma$ be the
logarithmic operator mean (a typical example of geodesic mean). Let $A,B$ be
non-zero matrices in $\bM_n^+$ such that their support projections are orthogonal. Then
$A\,\sigma\,B=0$ and so $\|A\,\sigma\,B\|_!=0$. But $\|A\|_!\,\sigma\,\|B\|_!>0$ since
$\|A\|_!>0$ and $\|B\|_!>0$. So it seems that Theorem \ref{T-4.7} is rather optimal.
\end{remark}

\vskip 5pt
To prove Theorem \ref{T-4.7} we need two lemmas.

\vskip 10pt
\begin{lemma}\label{L-4.10}
Let $A,B\in\bM_n^+$. If $A\prec^{w(\log)}B$, then $\|A\|_!\ge \| B\|_!$ for all derived
anti-norms.
\end{lemma}

\vskip 10pt
\begin{proof}
Let $p>0$ and assume that $A\prec^{w(\log)}B$ and $B$ is invertible. Then by assumption,
$A$ is also invertible and we have
$$
\prod_{j=1}^k\mu_j(A^{-p})=\Biggl\{\prod_{j=1}^k\mu_{n+1-j}(A)\Biggr\}^{-p}
\le\Biggl\{\prod_{j=1}^k\mu_{n+1-j}(B)\Biggr\}^{-p}
=\prod_{j=1}^k\mu_j(B^{-p})
$$
for all $k=1,\dots,n$, i.e., $A^{-p}\prec_{w(\log)}B^{-p}$. This implies that
$A^{-p}\prec_w B^{-p}$ and so $\|A^{-p}\|\le\|B^{-p}\|$ for any symmetric norm $\|\cdot\|$.
Therefore, $\| A^{-p} \|^{-1/p} \ge \| B^{-p} \|^{-1/p}$, which means that
$\Vert A\Vert_! \ge \Vert B\Vert_!$ for any derived anti-norm.
\end{proof} 

\vskip 10pt
\begin{lemma}\label{L-4.11}
Let $A,B\in\bM_n^+$ and $\alpha\in[0,1]$. Then
$\|A\,\#_{\alpha}\,B\|_! \ge \| A\|_!\,\#_{\alpha}\,\| B\|_!$ for all derived anti-norms.
\end{lemma}

\begin{proof}
The case $\alpha=0$ or $1$ is trivial. Assume that $0<\alpha<1$. Since
$\|A\|_!\,\#_\alpha\,\|B\|_!=0$ if $A$ or $B$ is not invertible, we may assume that both
$A$ and $B$ are invertible. The log-majorization \eqref{F-3.12} implies that
$$
A\,\#_\alpha\,B\prec^{w(\log)}A^\downarrow\,\#_\alpha\,B^\downarrow
$$
so that the previous lemma yields
$$
\| A\,\#_{\alpha}\,B \|_! \ge \| A^{\downarrow}\,\#_{\alpha}\,B^{\downarrow} \|_!
$$
for any derived anti-norm. To complete the proof, we need to show that 
\begin{equation}\label{F-4.3}
\| A^{\downarrow}\,\#_{\alpha}\,B^{\downarrow} \|_! \ge
\| A^{\downarrow}\|_!  \,\#_{\alpha} \,  \|B^{\downarrow} \|_!.
\end{equation}
This follows from the H\"older inequality for a symmetric gauge function $\Phi$:
If $q\in(1,\infty)$ and $1/q+1/r=1$, then
$$
\Phi(a_1b_1,\dots,a_nb_n)\le\Phi(a_1^q,\dots,a_n^q)^{1/q}\Phi(b_1^r,\dots,b_n^r)^{1/r}
$$
for every $a,b\in[0,\infty)^n$, see \cite[IV.1.6]{Bh}. From this, for every
$a,b\in(0,\infty)^n$ we have
\begin{align*}
\Phi((a_1^{1-\alpha}b_1^\alpha)^{-p},\dots,(a_n^{1-\alpha}b_n^\alpha)^{-p})
&=\Phi((a_1^{-p})^{1-\alpha}(b_1^{-p})^\alpha),\dots,
(a_n^{-p})^{1-\alpha}(b_n^{-p})^\alpha)) \\
&\le\Phi(a_1^{-p},\dots,a_n^{-p})^{1-\alpha}\Phi(b_1^{-p},\dots,b_n^{-p})^\alpha
\end{align*}
so that
$$
\Phi((a\,\#_\alpha\,b)^{-p})^{-1/p}\ge\Phi(a^{-p})^{-1/p}\,\#_\alpha\,\Phi(b^{-p})^{-1/p}.
$$
Hence \eqref{F-4.3} holds.
\end{proof}

\vskip 5pt
We turn to the proof of the theorem.

\bigskip\noindent
{\it Proof of Theorem \ref{T-4.7}.}\enspace Let $\sigma$ be a geodesic mean so that
$$
A\,\sigma \,B = \int_0^1 A\,\#_{\alpha}\,B\,d\nu(\alpha)
$$
with a probability measure $\nu$ on $[0,1]$.
From Theorem \ref{T-1.1} we infer that
$$
f\left(A\,\sigma\, B\right) \ge \frac{1}{2}\left\{
U\left(\int_0^1 f(A\,\#_{\alpha}\,B)\, d\nu(\alpha)\right)U^*
+
V\left(\int_0^1 f(A\,\#_{\alpha}\,B)\,d\nu(\alpha)\right)V^*
\right\}
$$
for some unitaries $U,\,V$. When $\nu$ is supported on a finite set, this directly follows
from Theorem \ref{T-1.1} since $f(t)$ is concave. When $\nu$ is a general probability
measure, we choose a sequence $\{\nu_l\}$ of finitely supported probability measures on
$[0,1]$ such that
$$
\int_0^1A\,\#_\alpha\,B\,d\nu_l(\alpha)\longrightarrow
\int_0^1A\,\#_\alpha\,B\,d\nu(\alpha),
$$
$$
\int_0^1f(A\,\#_\alpha\,B)\,d\nu_l(\alpha)\longrightarrow
\int_0^1f(A\,\#_\alpha\,B)\,d\nu(\alpha).
$$
One can then show the assertion by a simple convergence argument. Hence, by the concavity
property of anti-norms,
\begin{equation}\label{F-4.4}
\| f(A\,\sigma\, B) \|_! \ge \int_0^1 \|f(A\,\#_{\alpha}\,B)\|_!\,d\nu(\alpha).
\end{equation}

Next, from the log-majorization \eqref{F-3.12} and the fact that $f(t)$ is geometrically
concave, it is easy to see that
\begin{equation*}
f(A\,\#_{\alpha}\,B) \prec^{w(\log)}
f\bigl(A^{\downarrow}\,\#_{\alpha}B^{\downarrow}\bigr).
\end{equation*}
Thanks to the geometric concavity of $f(t)$ again we also have
\begin{equation*}
 f\bigl(A^{\downarrow}\,\#_{\alpha}\,B^{\downarrow}\bigr) \ge
 f(A^{\downarrow})\,\#_{\alpha}\,f(B^{\downarrow}),
\end{equation*}
which combined with the previous log-supermajorization yields
\begin{equation*}
f(A\,\#_{\alpha}\,B) \prec^{w(\log)} f(A^{\downarrow})\,\#_{\alpha}\,f(B^{\downarrow}).
\end{equation*}
Hence, for any derived anti-norm $\|\cdot\|_!$, Lemma \ref{L-4.10} implies that
\begin{equation*}
\|f(A\,\#_{\alpha}\,B) \|_!\ge \|  f(A^{\downarrow})\,\#_{\alpha}\,f(B^{\downarrow})\|_!,
\end{equation*}
which combined with Lemma \ref{L-4.11} yields 
\begin{equation*}
\|f(A\,\#_{\alpha}\,B) \|_!\ge
\| f(A^{\downarrow})\|_!\,\#_{\alpha}\,\|f(B^{\downarrow})\|_!
=\| f(A)\|_!\,\#_{\alpha}\,\|f(B)\|_!.
\end{equation*}
Inserting this into the integral inequality \eqref{F-4.4} completes the proof.\qed

\bigskip
We do not know whether Theorem \ref{T-4.7} can be generalized or not to the whole class
of operator means whose representing functions are geometrically convex, especially
whether it holds for the operator $p$-means $\sigma_p$ with $0\le p\le1$ (see Example
\ref{E-3.11}). However, it is possible to state a version of Theorem \ref{T-4.7} for the
power $p$-means $\beta_p$. This is a consequence of Theorem \ref{T-2.7}. A special case
was given in Corollary \ref{C-2.11}.

\vskip 10pt
\begin{prop}\label{P-4.12}
If $f(t)$ is a doubly concave function on  $\Omega$ and  $A,B\in\bM_n\{\Omega\}$, then
\begin{equation*}
\| f(A\,\beta_p\,B) \|_! \ge \|f(A)\|_! \,\beta_p\, \| f(B) \|_!
\end{equation*}
for all derived anti-norms and all power $p$-means $\beta_p$ with $0<p\le 1$. Moreover,
this holds for $p=0$ too when $A,B$ are invertible.
\end{prop}

\vskip 10pt
\begin{proof}
Assume that $0<p\le1$. By Theorem \ref{T-2.7}, arguing as for Corollary \ref{C-2.8},
we have 
\begin{equation*}
\|f(A\,\beta_p\,B) \|_! \ge \|f(A)\,\beta_p\,f(B)\|_!
\end{equation*}
for all symmetric anti-norms. It then suffices to show that, in case of an derived
anti-norm, one has
\begin{equation}\label{F-4.5}
\|X\,\beta_p\,Y\|_! \ge \| X\|_!\,\beta_p\,\| Y\|_!
\end{equation}
or equivalently,
$$
\left\|\left(\frac{X^p + Y^p}{2}\right)^{1/p}\right\|_! \ge
\left(\frac{\|X\|_!^p +\|Y\|_!^p }{2} \right)^{1/p}
$$
for all $X,Y\in \bM_n^+$. By taking the $p$th power of both sides, this is equivalent to
\begin{equation}\label{F-4.6}
\left\|(X^p + Y^p)^{1/p}\right\|_!^p \ge \|X\|_!^p +\|Y\|_!^p. 
\end{equation}
To check that \eqref{F-4.6} does hold, note that if $\|\cdot\|_!$ is derived from a
symmetric norm $\|\cdot\|$ and a scalar $r>0$, then $X\mapsto \| X^{1/p}\|_!^p$ is a
derived anti-norm from $\|\cdot \|$ and $pr>0$. Therefore, \eqref{F-4.6} and hence
\eqref{F-4.5} hold. The case $p=0$ is immediate by taking the limit from the case $0<p\le1$.
\end{proof}

For symmetric norms, we could expect a result similar to the previous proposition by using
Corollary \ref{C-2.14}.  But this is not possible: In general, if $g(t)$ is a doubly convex
function on $[0,\infty)$, $g(0)=0$, and $A,B\in\bM_n\{\Omega\}$, then neither
$
\| g(A\,\beta_q\,B) \| \ge \|g(A)\| \,\beta_q\, \| g(B) \|
$
nor its reversed inequality do hold for symmetric norms and power $q$-means $\beta_q$ with
$q>1$.

The last result of this section is the symmetric norm version of Theorem \ref{T-4.7}.
The proof is similar to that of Theorem \ref{T-4.7} by using the convex version of Theorem
\ref{T-1.1} and the symmetric norm versions of Lemmas \ref{L-4.10} and \ref{L-4.11};
namely, $A\prec_{w(\log)}B$ entails $\|A\|\le\|B\|$, and
$\|A\,\#_\alpha\,B\|\le\|A\|\,\#_\alpha\,\|B\|$ for every $A,B\in\bM_n^+$.

\vskip 10pt
\begin{prop}\label{P-4.13}
If $g(t)$ is a doubly convex function on $\Omega$ and $A,B\in\bM_n\{\Omega\}$, then
\begin{equation*}
\| g(A\,\sigma\, B) \| \le \|g(A)\|\,\sigma\, \| g(B) \|
\end{equation*}
for all symmetric norms $\|\cdot\|$ and all geodesic means $\sigma$.
\end{prop}

\section{Geodesic means for several matrices}

In this section we will extend geodesic means introduced in Section 3 to those for several
variables based on the Riemannian geometric approach in \cite{Mo,BH}. Let $\bP_n$ denote
the set of $n\times n$ positive definite matrices. It possesses a natural Riemannian
manifold structure and the induced geodesic distance is given as
$$
\delta(A,B)=\|\log A^{-1/2}BA^{-1/2}\|_2
=\Biggl\{\sum_{i=1}^n\log^2\lambda_i(A^{-1}B)\Biggr\}^{1/2},
\qquad A,B\in\bP_n,
$$
where $\lambda_i(A^{-1}B)$'s are the eigenvalues of $A^{-1}B$. Moreover, the geodesic path
joining $A,B$ is the weighted geometric means $A\,\#_t\,B$, $t\in[0,1]$. Note that this
$(\bP_n,\delta)$ is an example of so-called NPC spaces (nonpositively curved metric spaces),
whose theory has recently been developed extensively as seen in \cite{St}. Now, let
$\bw=(w_1,\dots,w_m)$ be a weight vector, i.e., $w_1,\dots,w_m\ge0$ and $\sum_{i=1}^mw_i=1$.
Given $m$ matrices $A_1,\dots,A_m\in\bP_n$, the {\it weighted geometric mean}
$G_m(\bw;A_1,\dots,A_m)$ is defined as a unique minimizer of the weighted sum of the
squares of distances, i.e.,
\begin{equation}\label{F-5.1}
G_m(\bw;A_1,\dots,A_m):=\mathop{\arg\min}_{X\in\bP_n}\sum_{i=1}^mw_i\delta^2(X,A_i),
\end{equation}
which is also called the {\it weighted least squares mean}, see \cite{LL}. The non-weighted
$m$-variable geometric mean is \eqref{F-5.1} when $\bw=(1/m,\dots,1/m)$. When $m=2$,
$G_2(1-t,t;A,B)=A\,\#_t\,B$ for $t\in[0,1]$ and $A,B\in\bP_n$. Below we will briefly write
$G_m(\bw;\bA)$ for $G_m(\bw;A_1,\dots,A_m)$ for $m$-tuples $\bA:=(A_1,\dots,A_m)$.

In \cite{LL} Lawson and Lim proved the monotonicity property of $G_m(\bw;\bA)$ by using
a powerful probabilistic tool in NPC spaces, see \cite[Theorem 4.7]{St}. The tool is
regarded as a sort of strong law of large numbers in NPC spaces, which will also be
crucial in our discussion below. So in the next lemma let us state it in a form
specialized to our purpose. For $A_1,\dots,A_m\in\bP_n$ the {\it inductive mean}
$S_m(A_1,\dots,A_m)$ is inductively defined as follows: $S_1(A_1):=A_1$ and
$S_k(A_1,\dots,A_k):=S_{k-1}(A_1,\dots,A_{k-1})\,\#_{1/k}\,A_k$ for $k=2,\dots,m$.

\begin{lemma}\label{L-5.1}
Let $A_1,\dots,A_m\in\bP_n$ and let $X_k:\Omega\to\bP_n$, $k\in\bN$, be a sequence of
i.i.d.\ random variables on a probability space $(\Omega,P)$ with distribution
$\sum_{i=1}^mw_i\delta_{A_i}$. Then $S_k(X_1(\omega),\dots,X_k(\omega))\to G_m(\bw;\bA)$
as $k\to\infty$ for almost every $\omega\in\Omega$.
\end{lemma}

A construction of the i.i.d.\ sequence $X_m$ in the corollary is easy: Let
$\Omega_0:=\{1,\dots,k\}$ with probability $P_0:=\sum_{i=1}^kw_i\delta_i$, and let
$(\Omega,P):=\prod_{m=1}^\infty(\Omega_0,P_0)$ be the infinite tensor product of
$(\Omega_0,P_0)$. Set $X_m(\omega):=A_{j_m}$ for $m\in\bN$ and
$\omega=(j_1,j_2,\dots)\in\Omega$.

To extend geodesic means for two matrices to those for $m$ matrices, let $\Sigma_m$ denote
the simplex of probability vectors on $m$ points, i.e.,
$\Sigma_m:=\{\bw=(w_1,\dots,w_m):w_i\ge0,\ \sum_{i=1}^mw_i=1\}$. For any probability
measure $\nu$ on $\Sigma_m$ we define for $A_1,\dots,A_m\in\bP_n$,
\begin{equation}\label{F-5.2}
\sigma_m(A_1,\dots,A_m)=\sigma_m(\bA):=\int_{\Sigma_m}G_m(\bw;\bA)\,d\nu(\bw),
\end{equation}
and call it an $m$-variable {\it geodesic mean}. In particular, with the uniform probability
measure $\nu_0$ on $\Sigma_m$ we define the $m$-variable {\it logarithmic mean} by
$$
L_m(A_1,\dots,A_m)=L_m(\bA):=\int_{\Sigma_m}G_k(\bw;\bA)\,d\nu_0(\bw),
$$
which extends the logarithmic mean $A\,\lambda\,B$ for two matrices since
$$
L_2(A,B)=\int_0^1A\,\#_t\,B\,dt=A\,\lambda\,B.
$$

\begin{prop}
Let $\sigma_m(\bA)$ be an $m$-variable geodesic mean defined in \eqref{F-5.2}. Then, for
every $A_1,\dots,A_m\in\bP_n$,
$$
\Biggl(\sum_{i=1}^m\overline{w}_iA_i^{-1}\Biggr)^{-1}\le
\sigma_m(\bA)\le\sum_{i=1}^m\overline{w}_iA_i,
$$
where $\overline{w}_i:=\int_{\Sigma_m}w_i\,d\nu(\bw)$, $1\le i\le m$. In particular,
$$
m\Biggl(\sum_{i=1}^mA_i^{-1}\Biggr)^{-1}\le
L_m(\bA)\le{1\over m}\sum_{i=1}^mA_i.
$$
\end{prop}

\begin{proof}
It was proved in \cite{LL} that
$$
\Biggl(\sum_{i=1}^mw_iA_i^{-1}\Biggr)^{-1}\le
G_m(\bw;\bA)\le\sum_{i=1}^mw_iA_i.
$$
Integrating over $\Sigma_m$ by $\nu$ we have
$$
\int_{\Sigma_m}\Biggl(\sum_{i=1}^mw_iA_i^{-1}\Biggr)^{-1}\,d\nu(\bw)\le
\sigma_m(\bA)\le\int_{\Sigma_m}\sum_{i=1}^mw_iA_i\,d\nu(\bw).
$$
It is obvious that
$$
\int_{\Sigma_m}\sum_{i=1}^mw_iA_i\,d\nu(\bw)=\sum_{i=1}^m\overline{w}_iA_i.
$$
Since $x^{-1}$ ($x>0$) is operator convex,
$$
\int_{\Sigma_m}\Biggl(\sum_{i=1}^mw_iA_i^{-1}\Biggr)^{-1}\,d\nu(\bw)
\ge\Biggl(\int_{\Sigma_m}\sum_{i=1}^mw_iA_i^{-1}\,d\nu(\bw)\Biggr)^{-1}
=\Biggl(\sum_{i=1}^m\overline{w}_iA_i^{-1}\Biggr)^{-1}.
$$
\end{proof}

The above proposition says that an $m$-variable geodesic mean is between the $m$-variable
weighted harmonic and arithmetic means. The naturally expected inequality
$G_m(\bA)\le L_m(\bA)$ is not known, where $G_m(\bA)$ is the non-weighted geometric mean
$G_m(\bw;\bA)$ with $\bw=(1/m,\dots,1/m)$.

We now prove the log-majorization for the weighted geometric mean $G_m(\bw;\bA)$.

\begin{prop}\label{P-5.3}
For every $\bw\in\Sigma_m$ and every $A_1,\dots,A_m\in\bP_n$,
\begin{equation}\label{F-5.3}
G_m(\bw;\bA)\prec_{(\log)}G_m(\bw;\bA^\downarrow),
\end{equation}
where $G_m(\bw;\bA^\downarrow)$ stands for $G_m(\bw;A_1^\downarrow,\dots,A_m^\downarrow)$.
\end{prop}

\begin{proof}
By Corollary \ref{C-3.6} we have for every $X_1,X_2,\dots\in\bP_n$ and every $k\ge2$,
$$
S_k(X_1,\dots,X_k)\prec_{(\log)}
S_{k-1}(X_1,\dots,X_{k-1})^\downarrow\,\#_{1/k}\,X_k^\downarrow.
$$
Iterating this for $k=2,3,\dots$ we notice that
$$
S_k(X_1,\dots,X_k)\prec_{(\log)}S_k(X_1^\downarrow,\dots,X_k^\downarrow)
$$
for every $k\in\bN$. Let $X_k:\Omega\to\bP_n$, $k\in\bN$, be as in Lemma \ref{L-5.1}
associated with given $\bw$ and $A_1,\dots,A_m\in\bP_n$. We then have
\begin{equation}\label{F-5.4}
S_k(X_1(\omega),\dots,X_k(\omega))\prec_{(\log)}
S_k(X_1(\omega)^\downarrow,\dots,X_k(\omega)^\downarrow)
\end{equation}
for all $\omega\in\Omega$. Note that $X_k(\omega)^\downarrow$, $k\in\bN$, is a sequence of
i.i.d.\ random variables with distribution $\sum_{i=1}^mw_i\delta_{A_i^\downarrow}$.
Lemma \ref{L-5.1} implies that both sides of \eqref{F-5.4} converge to those of
\eqref{F-5.3}, respectively, as $k\to\infty$ for almost every $\omega$. Hence \eqref{F-5.3}
holds.
\end{proof}

The next result is the $m$-variable extension of Proposition \ref{P-3.7}. The proof
based on Proposition \ref{P-5.3} is similar to that of Proposition \ref{P-3.7}.

\begin{prop}\label{prop-5-2}
Let $\sigma_m$ be an $m$-variable geodesic mean and $A_1,\dots,A_m\in\bP_n$. Then, for
every $k=1,\dots,n$,
$$
\Biggl\{\prod_{j=n+1-k}^n\mu_j(\sigma_m(\bA))\Biggr\}^{1/k}
\ge\sigma_m\Biggl(\Biggl\{\prod_{j=n+1-k}^n\mu_j(A_1)\Biggr\}^{1/k},
\dots,\Biggl\{\prod_{j=n+1-k}^n\mu_j(A_m)\Biggr\}^{1/k}\Biggr).
$$
\end{prop}

\vskip 10pt
Furthermore, in the next theorem we similarly have the $m$-variable versions of Theorem
\ref{T-4.7} and Proposition \ref{P-4.13}. The proof  is similar to that in
Section 4.

\vskip 10pt
\begin{theorem}\label{T-5.5}
Let $\sigma_m$ be an $m$-variable geodesic mean and let $A_1,\dots,A_m\in\bM_n\{\Omega\}$
for an interval $\Omega\subset(0,\infty)$.
\begin{itemize}
\item[\rm{1}.]
If $f(t)$ is a doubly concave function on $\Omega$, then
\begin{equation*}
\| f(\sigma_m(\bA)) \|_! \ge \sigma_m( \|f(\bA)\|_!)
\end{equation*}
for all derived anti-norms $\|\cdot\|_!$, where $\sigma_m(\|f(\bA)\|_!):=
\sigma_m(\|f(A_1)\|_1,\dots,\allowbreak\|f(A_m)\|_!)$.
\item[\rm{2}.]
If $g(t)$ is a doubly convex function on $\Omega$, then
\begin{equation*}
\| g(\sigma_m(\bA)) \| \le \sigma_m( \|g(\bA)\|)
\end{equation*}
for all symmetric norms $\|\cdot\|$, where $\sigma_m(\|g(\bA)\|)$ is as above.
\end{itemize}
\end{theorem}

\vskip 10pt
A particular case of the second assertion of the theorem with $g(t)=t$ (or rather a
consequence of Proposition \ref{P-5.3}) is a very recent inequality for the weighted
geometric mean due to Bhatia and Karandikar \cite{BK}:
$$
\| G_m(\bw;\bA)\| \le \prod_{i=1}^m \| A_i\|^{w_i}.
$$
For derived anti-norms, the reverse inequality holds.

\section{Miscellaneous results on anti-norms}

This section gives some additional results on anti-norms while we have not used them in
the main body of the paper. The first proposition is concerned with duality of anti-norms.

\begin{prop}
Let $\|\cdot\|_!$ be a symmetric anti-norm on $\bM_n^+$ assumed not identically zero, and
define for every $A\in\bM_n^+$,
$$
\|A\|_!':=\inf\{\Tr AB: B\in\bM_n^+,\,\|B\|_!=1\}.
$$
Then $\|\cdot\|_!'$ is a symmetric anti-norm on $\bM_n^+$ too.
\end{prop}

\begin{proof}
The properties 1--3 of Definition \ref{D-4.1} for $\|\cdot\|_!'$ are immediate from
definition. To prove continuity, let $\{A_l\}$ be a sequence in $\bM_n^+$ converging to
$A\in\bM_n^+$. For any $B\in\bM_n^+$ with $\|B\|_!=1$, since
$\|A_l\|_!'\le\Tr A_lB\to\Tr AB$, we have $\limsup_{l\to\infty}\|A_l\|_!'\le\Tr AB$ so that
$\limsup_{l\to\infty}\|A_l\|_!'\le\|A\|_!'$. To show that
$\|A\|_!'\le\liminf_{l\to\infty}\|A_l\|_!'$, let $\Phi_!$ be the symmetric anti-gauge
function corresponding to $\|\cdot\|_!$, see \cite[Proposition 3.2]{BH1}. From the fact
that $\Tr AB\ge\Tr A^\downarrow B^\uparrow$ for $A,B\in\bM_n^+$, it is easy to see that
$\|X\|_!'=\Phi_!'(\mu(X))$ for all $X\in\bM_n^+$, where
$$
\Phi_!'(x):=\inf\Biggl\{\sum_{i=1}^nx_iy_i:y\in\bR_+^n,\,\Phi_!(y)\ge1\Biggr\},
\qquad x\in\bR_+^n.
$$
Since $\mu(A_l)\to\mu(A)$, we need to show that
$\Phi_!'(a)\le\liminf_{l\to\infty}\Phi_!'(a^{(l)})$ if $a^{(l)}\to a$ in $\bR_+^n$. For
each $l$ choose a $b^{(l)}\in\bR_+^n$ such that $\Phi_!(b^{(l)})\ge1$ and
\begin{equation}\label{F-6.1}
\sum_{i=1}^na_i^{(l)}b_i^{(l)}<\Phi_!'(a^{(l)})+l^{-1}.
\end{equation}
By taking a subsequence we may assume that $b^{(l)}\to b\in[0,\infty]^n$. Moreover, since
$\Phi_!$ is continuous and monotone (i.e., $\Phi_!(x)\le\Phi_!(y)$ if $x\le y$ in
$\bR_+^n$), one can extend $\Phi_!$ to a continuous functional on $[0,\infty]^n$ with
values in $[0,\infty]$. By \eqref{F-6.1} it follows that
$$
\sum_{i=1}^na_ib_i\le\liminf_{l\to\infty}\sum_{i=1}^na_i^{(l)}b_i^{(l)}
\le\liminf_{l\to\infty}\Phi_!'(a^{(l)}),
$$
with convention $a_ib_i:=0$ for $a_i=0$ and $b_i=\infty$. Since $\Phi_!(b)\in[1,\infty]$
by continuity, one can choose $\tilde b^{(m)}\in\bR_+^n$, $m\in\bN$, such that
$\tilde b^{(m)}\le b$ and $0<\beta_m:=\Phi_!(\tilde b^{(m)})\to\Phi_!(b)$. Therefore,
$$
\Phi_!'(a)\le\sum_{i=1}^na_i(\beta_m^{-1}b_i^{(m)})
\le\beta_m^{-1}\sum_{i=1}^na_ib_i.
$$
Letting $m\to\infty$ yields that $\Phi_!'(a)\le\liminf_{l\to\infty}\Phi_!'(a^{(l)})$.
\end{proof}

We call the above $\|\cdot\|_!'$ the {\it dual anti-norm} of $\|\cdot\|_!$. It is plain to
show that the dual anti-norm of $\|\cdot\|_!'$ goes back to $\|\cdot\|_!$ like symmetric
norms.

\begin{example}
When $p\in(0,1)$ and $q\in(-\infty,0)$ with $1/p+1/q=1$, the reverse H\"older inequality
shows that
$$
\|a\|_q=\inf\Biggl\{\sum_{i=1}^na_ib_i:b\in\bR_+^n,\,\|b\|_p=1\Biggr\}
$$
for every $a\in\bR_+^n$, where $\|a\|_q$ and $\|b\|_p$ are defined for vectors in $\bR_+^n$
as in Example \ref{E-4.3}. This implies that the Schatten anti-norm $\|\cdot\|_p$ and the
negative Schatten anti-norm $\|\cdot\|_q$ (Example \ref{E-4.3}) on $\bM_n^+$ are the dual
of each other. Letting $p\searrow0$ (and $q\nearrow0$) we observe that
$A\mapsto\det^{1/n}A$ is dual to itself up to a constant; in fact, this is verified from
$$
\inf\Biggl\{\sum_{i=1}^na_ib_i:b\in\bR_+^n,\,(b_1\cdots b_n)^{1/n}\ge1\Biggr\}
=n(a_1\cdots a_n)^{1/n},\qquad a\in\bR_+^n.
$$
Thus, the Minkowski functional is special as the self-dual symmetric anti-norm, likewise
the Hilbert-Schmidt norm is a special symmetric norm. More generally, it is worthwhile to
note that the correspondence $\|\cdot\|_p\mapsto\|\cdot\|_{2-p}$, $p\in[1,\infty]$,
transforms the Schatten norms to the Schatten anti-norms and preserves the duality paring.
Here, $\|\cdot\|_{-\infty}$ means $\|\cdot\|_{\{1\}}$, i.e., the functional taking the
smallest eigenvalue $\lambda_n(A)$, which is the dual anti-norm of the trace on $\bM_n^+$.
\end{example}

\vskip 5pt
In the next proposition we give two expressions for the Ky Fan $k$-anti-norms.

\begin{prop}\label{P-6.3}
For every $Z\in\bM_n^+$ and every $k=1,\dots,n$,
\begin{align}
\|Z\|_{\{k\}}
&=\min\{\Tr ZP: P\ \mbox{is a projection},\ \mathrm{rank}\,P=k\} \label{F-6.2}\\
&=\max\{k\lambda_n(A)-\Tr B:A,B\in\bM_n^+,\,Z=A-B\}. \label{F-6.3}
\end{align}
\end{prop}

\begin{proof}
The proof of \eqref{F-6.2} is well-known and parallel to that of the similar expression
for the Ky Fan norms. To prove \eqref{F-6.3}, notice that for any $A,B\in\bM_n^+$ with
$Z=A-B$ we have
\begin{align*}
\|Z\|_{\{k\}}&=\sum_{j=n+1-k}^n\lambda_j(A-B) \\
&\ge\sum_{j=n+1-k}^n\lambda_j(A)-\sum_{j=1}^k\lambda_j(B)
\ge k\lambda_n(A)-\Tr B.
\end{align*}
Let $Z=\sum_{j=1}^n\lambda_j(Z)P_j$ be the spectral decomposition with orthogonal
projections $P_j$ of rank $1$. Set
\begin{align*}
A&:=\sum_{j=1}^{n-k}\lambda_j(Z)P_j+\lambda_{n+1-k}(Z)\sum_{j=n+1-k}^nP_j, \\
B&:=\sum_{j=n+1-k}^n\{\lambda_{n+1-k}(Z)-\lambda_j(Z)\}P_j.
\end{align*}
We then have $Z=A-B$ and
$$
\lambda_n(A)=\lambda_{n+1-k}(Z),\qquad
\Tr B=k\lambda_{n+1-k}(Z)-\sum_{j=n+1-k}^n\lambda_j(Z)
$$
so that $k\lambda_n(A)-\Tr B=\|Z\|_{\{k\}}$.
\end{proof}

\vskip 5pt
The expression \eqref{F-6.3} is considered as a kind of K-functional in the real
interpolation theory. In fact, thanks to \cite[Lemma 4.2]{BH1} that reduces the proof to
the Ky Fan $k$-anti-norms, \eqref{F-6.3} gives the anti-norm counterpart to a familiar
interpolation property of symmetric norms though the assumptions (unitality and
trace-preservation) on $\bE$ seem too strict.
Note that this can alternatively be proved as follows: If $\bE$ and $Z$ are
as in the next corollary, then we have $\bE(Z)\prec Z$, which implies that
$\bE(Z)\prec^{w}Z$  and hence
$\|E(Z)\|_!\ge\|Z\|_!$ by \cite[Lemma 4.2]{BH1}.

\vskip 5pt
\begin{cor}\label{C-6.4}
Let $\bE:\bM_n\to\bM_n$ be a positive linear map and assume that $\bE$ is unital and
trace-preserving. Then $\|\bE(Z)\|_!\ge\|Z\|_!$ holds for all $Z\in\bM_n^+$ and all
symmetric anti-norms.
\end{cor}

\vskip 3pt
For instance, when $A\in\bM_n^+$ has diagonal entries all equal to $1$, the
Schur multiplication map $\bE(Z):=A\circ Z$ satisfies the assumption of the above corollary.
In fact, the result can be improved in this situation as follows:

\vskip 6pt
\begin{theorem}\label{T-6.5}
If $A\in\bM_n^+$ has   diagonal entries  all greater than or equal to $1$, then 
$$\|A\circ Z\|_!\ge\|Z\|_!$$
 holds for all $Z\in\bM_n^+$ and all symmetric anti-norms.
\end{theorem}

\vskip 5pt
\begin{proof}
Let $D$ be the diagonal matrix with the same diagonal part as $A$; then $D\ge I$ by
assumption. It is immediate to notice that
$D^{-1/2}(A\circ Z)D^{-1/2}=(D^{-1/2}AD^{-1/2})\circ Z$ for all $Z\in\bM_n$ and the
diagonal entries of $D^{-1/2}AD^{-1/2}$ are all equal to $1$. For every symmetric
anti-norm $\|\cdot\|_!$ and every $Z\in\bM_n^+$ we have
$$
\|A\circ Z\|_!\ge\|D^{-1/2}(A\circ Z)D^{-1/2}\|_!
=\|(D^{-1/2}AD^{-1/2})\circ Z\|_!\ge\|Z\|_!,
$$
where the first inequality follows from the fact that $\|XYX\|_!\le\|Y\|_!$ for any
$X,Y\in\bM_n^+$ with $X\le I$ (since $(XYX)^\downarrow\le Y^\downarrow$) and the second
inequality is a special case of Corollary \ref{C-6.4}.
\end{proof}

The above theorem also contains some trace inequalities. Indeed, Theorem \ref{T-6.5} means 
that we have the supermajorization $A\circ Z \prec^w Z$. Since concave increasing functions
preserve supermajorization we infer:

\vskip 5pt
\begin{cor}
Let $A\in\bM_n^+$ with all its diagonal entries  greater than or equal to $1$. Then, for
every increasing concave function $f(t)$ on $[0,\infty)$ and every $Z\in\bM_n^+$,
$$
{\mathrm{Tr\,}} f(A\circ Z)\ge {\mathrm{Tr\,}} f(Z).
$$ 
\end{cor}

\vskip 5pt
In the following we apply Theorem \ref{T-6.5} to obtain the anti-norm version of the
arithmetic-geometric inequality.

\begin{cor}
If $A\in\bM_n^+$ is invertible and $0<\alpha\le1/2$, then, for any symmetric anti-norm
$\|\cdot\|_!$ and every $Z\in\bM_n^+$,
\begin{equation}\label{F-6.4}
\|Z\|_!\ge\frac{1}{2\alpha}\,\bigg\|\int_0^\alpha
(A^{t-\frac{1}{2}}ZA^{\frac{1}{2}-t}+A^{\frac{1}{2}-t}ZA^{t-\frac{1}{2}})\,dt\bigg\|_!,
\end{equation}
whenever the matrix integral in the right-hand side is in $\bM_n^+$.
\end{cor}

\begin{proof}
We may assume that $A$ is a diagonal matrix with diagonals $\lambda_1,\dots,\lambda_n$.
Then it is easy to check that
$$
\frac{1}{2\alpha}\int_0^\alpha
(A^{t-\frac{1}{2}}ZA^{\frac{1}{2}-t}+A^{\frac{1}{2}-t}ZA^{t-\frac{1}{2}})\,dt
=\Biggl[\frac{\int_0^\alpha(\lambda_i^t\lambda_j^{1-t}+\lambda_i^{1-t}\lambda_j^t)\,dt}
{2\alpha\sqrt{\lambda_i\lambda_j}}\Biggr]_{i,j=1}^n\circ Z.
$$
Hence \eqref{F-6.4} follows from Theorem \ref{T-6.5} once we show that
$$
\Biggl[\frac{2\alpha\sqrt{\lambda_i\lambda_j}}
{\int_0^\alpha(\lambda_i^t\lambda_j^{1-t}+\lambda_i^{1-t}\lambda_j^t)\,dt}
\Biggr]_{i,j=1}^n \in\bM_n^+.
$$
For this it suffices by \cite[Theorem 1.1]{HK0} to prove that
$$
\phi(x):=\frac{2\alpha e^x}{\int_0^\alpha(e^{2tx}+e^{2(1-t)x})\,dt}
$$
is a positive definite function on $\bR$. A direct computation gives
\begin{align*}
\phi(x)&=\frac{4\alpha x}{e^x-e^{-x}-e^{(1-2\alpha)x}+e^{(2\alpha-1)x}} \\
&=\frac{2\alpha x}{\sinh x-\sinh((1-2\alpha)x)} \\
&=\frac{\alpha x}{\cosh((1-\alpha)x)\sinh(\alpha x)}.
\end{align*}
Since $1/\cosh((1-\alpha)x)$ and $x/\sinh(\alpha x)$ are positive definite on $\bR$ (see
\cite{HK0}), so is $\phi(x)$.
\end{proof}

In particular, letting $\alpha\searrow0$ in \eqref{F-6.4} gives
\begin{equation}\label{F-6.5}
\|Z\|_!\ge\frac{1}{2}\|A^{1/2}ZA^{-1/2}+A^{-1/2}ZA^{1/2}\|_!
\end{equation}
whenever $A^{1/2}ZA^{-1/2}+A^{-1/2}ZA^{1/2}\ge0$. Moreover, the case $\alpha=1/2$ of
\eqref{F-6.4} is
\begin{equation}\label{F-6.6}
\|Z\|_!\ge\bigg\|\int_0^1A^{t-\frac{1}{2}}ZA^{\frac{1}{2}-t}\,dt\bigg\|_!
\end{equation}
whenever $\int_0^1A^{t-\frac{1}{2}}ZA^{\frac{1}{2}-t}\,dt\ge0$.
Observe that \eqref{F-6.5} can be written as 
\begin{equation*}
\|A^{1/2}(A^{-1/2}ZA^{-1/2})A^{1/2}\|_!
\ge\frac{1}{2}\|A(A^{-1/2}ZA^{-1/2})+(A^{-1/2}ZA^{-1/2})A\|_!
\end{equation*}
whenever $A(A^{-1/2}ZA^{-1/2})+(A^{-1/2}ZA^{-1/2})A\ge0$. Hence we obtain the next
corollary.

\begin{cor}\label{rev-AGM} Let $A, Z\in\bM_n^+$ and assume that $AZ + ZA \ge 0$. Then, for
any symmetric anti-norm,
\begin{equation*}
\| A^{1/2}ZA^{1/2}\|_! \ge \frac{1}{2} \| AZ + ZA \|_!.
\end{equation*}
\end{cor}

\vskip 5pt
This may be considered as the anti-norm counterpart of the arithmetic-geometric
inequality $\frac{1}{2}\|AZ+ZA\|\ge\|A^{1/2}ZA^{1/2}\|$ for symmetric norms. Similarly,
\eqref{F-6.6} is the anti-norm counterpart of the logarithmic-geometric inequality
 for symmetric norms, see \cite{HK0} for
symmetric norm inequalities for means of matrices. The special case of the Minkowski
functional in Corollary \ref{rev-AGM} yields the well-known determinantal inequality
$$
\det A\det Z \ge \det \left(\frac{AZ +ZA}{2}\right) \
$$
whenever $A,Z\in \bM_n^+$ and $AZ + ZA \ge 0$.
The positivity assumption $AZ + ZA \ge 0$ is essential; in fact
$\det ((AZ +ZA)/2) \le \det A\det Z$ does not hold for all $A, Z\in\bM_n^+$, as it is shown
by considering
$$
A=\begin{pmatrix} 1&0\\ 0&0 \end{pmatrix} \oplus \begin{pmatrix} 1&0\\ 0&0 \end{pmatrix},
\quad 
Z=\begin{pmatrix} 1 &1\\ 1&1 
\end{pmatrix}
\oplus\begin{pmatrix} 1 &1\\ 1&1 
\end{pmatrix} .
$$

\vskip 5pt
When $\|\cdot\|$ is a symmetric norm on $\bM_n$, an extended version of the matrix H\"older
inequality
\begin{equation}\label{F-6.11}
\|AB\|\le\|\,|A|^p\,\|^{1/p}\|\,|B|^q\,\|^{1/q},\qquad A,B\in\bM_n,
\end{equation}
is well-known \cite[IV.2.6]{Bh}, where $p,q\in(1,\infty)$ with $1/p+1/q=1$. On the other
hand, the matrix reverse H\"older inequality
\begin{equation}\label{F-6.12}
\Tr AB\ge\|A\|_p\|B\|_q,\qquad A,B\in\bM_n^+,
\end{equation}
was very recently noticed in \cite{HMPB}, where $p\in(0,1)$ and $q\in(-\infty,0)$ with
$1/p+1/q=1$. Here, note that $\|A\|_p$ and $\|B\|_q$ are the Schatten anti-norms. Similarly
to \eqref{F-6.11} we extend \eqref{F-6.12} to a reverse H\"older inequality involving a
derived anti-norm ($\|B^q\|^{1/q}$ in \eqref{F-6.13}) introduced in Proposition \ref{P-4.6}.

\begin{prop}
Let $\|\cdot\|$ be a symmetric norm on $\bM_n$ and let $p\in(0,1)$ and $q\in(-\infty,0)$
with $1/p+1/q=1$. Then, for every $A,B\in\bM_n^+$,
\begin{equation}\label{F-6.13}
\|AB\|\ge\|A^p\|^{1/p}\|B^q\|^{1/q}.
\end{equation}
\end{prop}

\vskip 5pt
\begin{proof}
We may assume that $A$ and $B$ are invertible. By the Gel'fand and Naimark majorization
we have
$$
|AB|\succ_{(\log)}A^\downarrow B^\uparrow
$$
so that $\|AB\|\ge\|A^\downarrow B^\uparrow\|$. It is elementary to check that
$$
st\ge\frac{s^p}{p}+\frac{t^q}{q},\qquad s,t>0.
$$
This implies that
$p^{-1}(A^\downarrow)^p\le A^\downarrow B^\uparrow+(-q)^{-1}(B^\uparrow)^q$ and hence
$p^{-1}\|A^p\|\le\|A^\downarrow B^\uparrow\|+(-q)^{-1}\|B^q\|$ so that
$$
\|A^\downarrow B^\uparrow\|\ge\frac{\|A^p\|}{p}+\frac{\|B^q\|}{q}.
$$
Replacing $A$, $B$ with $\alpha A$, $\alpha^{-1}B$ for any $\alpha>0$ we have
$$
\|A^\downarrow B^\uparrow\|\ge\frac{\alpha^p\|A^p\|}{p}+\frac{\alpha^q\|B^q\|}{q}.
$$
Maximizing the above right-hand side over $\alpha>0$ yields
$$
\|A^\downarrow B^\uparrow\|\ge\|A^p\|^{1/p}\|B^q\|^{1/q},
$$
and the required inequality follows.
\end{proof}

As noticed in the above proof we have the inequality
$\|AB\| \ge \| A^{\downarrow} B^{\uparrow}\|$ for every $A,B\in\bM_n^+$ and all symmetric
norms. Thanks to the Araki (see \cite{Bh}) and the Gel'fand-Naimark log-majorizations,
this is refined for every $r\in(0,1)$ as 
\begin{equation}\label{F-6.14}
\|\,|A^{1/r}B^{1/r}|^r \| \ge \|AB\| \ge  \|\,|A^rB^r|^{1/r} \|
\ge \| A^{\downarrow} B^{\uparrow}\|.
\end{equation}
Hence the left-hand side of \eqref{F-6.13} can be replaced by $ \|\,|A^rB^r|^{1/r} \| $,
in particular by $ \| A^{1/2}BA^{1/2}\|$. By letting $r\searrow 0$ we also obtain the
following result which can be regarded as a substitute for \eqref{F-3.6} in case of the
mean $\beta_0$ in \eqref{F-2.1}.

\begin{cor}
For every invertible $A,B \in \bM_n^+$ and every $k=1,\dots,n$,
$$
{1\over k}\sum_{j=1}^k\mu_j(A\,\beta_0\,B)
\ge \Biggl\{\prod_{j=1}^k\mu_j(A)\Biggr\}^{1/2k}
\Biggl\{\prod_{j=1}^k\mu_{n+1-j}(B)\Biggr\}^{1/2k}.
$$
\end{cor}

\vskip 5pt
\begin{proof}
As remarked above by \eqref{F-6.14} we have 
$$
\|\, |A^rB^r|^{1/r} \| \ge \| A^p\|^{1/p}\| B^q\|^{1/q}
$$
for all symmetric norms, $r>0$, and $1/p +1/ q=1$ with $p\in(0,1)$. The Lie-Trotter formula
(see \cite{Bh}) says that $\lim_{r\searrow 0} |A^rB^r|^{1/r}=\exp(\log A + \log B)$, and thus
$$
\|\exp(\log A + \log B)\| \ge \| A^p\|^{1/p}\| B^q\|^{1/q}.
$$
Letting $\|\cdot\|=k^{-1}\|\cdot\|_{(k)}$
($\|\cdot\|_{(k)}$ being the Ky Fan norm) and $p\searrow0$ ($q\nearrow0$)
we obtain
$$
{1\over k}\sum_{j=1}^k\mu_j(\exp(\log A + \log B))
\ge \Biggl\{\prod_{j=1}^k\mu_j(A)\Biggr\}^{1/k}
\Biggl\{\prod_{j=1}^k\mu_{n+1-j}(B)\Biggr\}^{1/k}
$$
for $k=1,\dots,n$. The result follows by replacing $A,B$ by $A^{1/2}, B^{1/2}$.
\end{proof}

\section{Concluding remarks}

Majorization, symmetric norms and their connection with convex/concave functions  play
important roles in matrix analysis. In this paper the stress falls on supermajorization
and anti-norms. It seems that they provide a good framework to study Jensen type
inequalities for operator means. Many questions remain open. Some of them have been
noticed in the text. The binomial operator means $\sigma_{1/m}$ (Example \ref{E-3.11})
and the logarithmic operator mean are cases of special interest deserving
further investigation. It would be also of interest to characterize anti-norms preserving
the log-supermajorization order.

Although we confine our study to the matrix case, it is possible, with some slight
variations, to develop a theory of anti-norms for compact operators and for type II
factors with a finite trace (the semi-finite case might be more delicate).

\vskip 20pt

{\small
Jean-Christophe Bourin

Laboratoire de math\'ematiques, Universit\'e de Franche Comt\'e,

25030 Besan\c con, France

jcbourin@univ-fcomte.fr

\vskip 20pt
Fumio Hiai

Graduate School of Information Sciences, Tohoku University,

Aoba-ku, Sendai 980-8579, Japan

fumio.hiai@gmail.com
}

\end{document}